\documentclass[11pt,twoside]{amsart}
\usepackage{amsfonts,amsmath,amstext,amsbsy,amsthm,amscd,amsmath,amssymb,dsfont}
\usepackage[utf8]{inputenc}
\usepackage{hyperref}
\usepackage[T1]{fontenc}

\newcommand{\norme}[1]{\left\Vert #1\right\Vert}
\newtheorem{Lemme}{Lemma}[section]
\newtheorem{Prop}{Proposition}[section]  
\newtheorem{Def}{Definition}[section]
\newtheorem{Rmq}{Remark}[section]
\newtheorem{Thm}{Theorem}[section]

\theoremstyle{remark}

\usepackage{indentfirst}
\newcommand{\be}{\begin{equation}}
\newcommand{\ee}{\end{equation}}
\newcommand{\ba}{\begin{array}}
\newcommand{\ea}{\end{array}}
\newcommand{\bea}{\begin{eqnarray}}
\newcommand{\eea}{\end{eqnarray}}
\newcommand{\bee}{\begin{eqnarray*}}
\newcommand{\eee}{\end{eqnarray*}}

\renewcommand{\div}{\mbox{\rm div}\;\!}

\newcommand{\B} {\mathbb{B}}
\newcommand{\C} {\mathbb{C}} 
\newcommand{\N} {\mathbb{N}}
\newcommand{\R} {\mathbb{R}}    
\renewcommand{\S} {\mathbb{S}}    
\newcommand{\Z} {\mathbb{Z}} 
\newcommand{\cA} {\mathcal{A}} 
\newcommand{\cB} {\mathcal{B}}    
\newcommand{\cC} {\mathcal{C}}

\newcommand{\cH} {\mathcal{H}}     
\newcommand{\cI} {\mathcal{I}}     
     
\newcommand{\cL} {\mathcal{L}}      
\newcommand{\cM} {\mathcal{M}}  
\newcommand{\cN} {\mathcal{N}}  
\newcommand{\cO} {\mathcal{O}}     
\newcommand{\cR} {\mathcal{R}}

\def \with {\quad\!\hbox{with}\!\quad}
\def \andf {\quad\!\hbox{and}\!\quad}

\def\dD{\delta\!D}
\def\dZ_1{\delta\!Z_1}

\def\e{\varepsilon}
\def\d{\partial}

\def\wh{\widehat}
\def\wt{\widetilde}

\def\ddj{\dot\Delta_j}

\title[Partially dissipative hyperbolic systems]{Global existence  for partially dissipative hyperbolic systems in the $\textsc{L}^p$ framework, and relaxation limit}

\date{}


\author{Timothée Crin-Barat and Raphaël Danchin}

\begin{document}


\begin{abstract}
   Here we investigate global strong solutions for a class of partially dissipative hyperbolic systems in the framework of critical homogeneous Besov spaces. Our primary goal is to extend the analysis of our previous paper \cite{CBD2} to a functional framework
   where the low frequencies of the solution are only bounded in $L^p$-type spaces with $p$ larger than $2.$ This enables us  to prescribe weaker smallness conditions for global well-posedness
   and  to get a more accurate information on the qualitative properties of the constructed solutions. 
   Our existence theorem  in particular applies 
   to the multi-dimensional isentropic compressible Euler system
   with relaxation, and provide us with bounds that are 
   \emph{independent} of the relaxation parameter. As a consequence, 
we justify rigorously the  relaxation limit 
to  the porous media equation and exhibit explicit rates of convergence for suitable norms, a completely new result to the
best of our knowledge.
 \end{abstract}

\maketitle

\section*{Introduction}

We are concerned with  multi-dimensional first order  $n$-component systems in $\R^d$ for $d\geq1$ of the type:
\begin{equation}
\frac{\partial V}{\partial t} + \sum_{k=1}^dA^k(V)\frac{\partial V}{\partial x_k}+\dfrac{LV}{\varepsilon}=0 \label{GEQSYM}
\end{equation}
where $\varepsilon$ stands for a positive relaxation parameter
and the unknown $V=V(t,x)\in\R^n$ depends on the time variable 
$t\in \mathbb{R}_+$ and  on the space 
variable $x\in\mathbb{R}^d.$ 
The symmetric matrices valued maps 
$A^k$  ($k=1,\cdots,d$)  are assumed to be linear. 
We suppose that  \emph{partial} dissipation occurs, that is to say,
the $n\times n$ matrix $L$ is such that 
$L\!+\!{}^TL$ is nonnegative, but not necessarily definite positive 
(in other words, the term $LV$ concerns only  on a part of the solution).
In order to achieve global-in-time results, we shall 
suppose that the so-called  Shizuta-Kawashima condition 
is satisfied (see details in the next section) so 
that dissipation acts -- indirectly -- on all the components
of the solution. 
\smallbreak
We supplement System \eqref{GEQSYM} with an initial data $V_0$ at time $t=0$ and are concerned with the existence of global strong solutions in the case where $V_0$ is close 
to some constant state $\bar V$ such that $L\bar V=0,$
and to the relaxation limit $\varepsilon\to0.$
\smallbreak
Although our structure assumptions on \eqref{GEQSYM} 
 will be  a bit more restrictive than in  \cite{CBD2}, 
 our approach applies to  the following 
 compressible Euler system with relaxation: 
\begin{equation} \left\{ \begin{matrix}\partial_t \rho +\text{div}(\rho v)=0,\\[1ex] \partial_t (\rho v)+\div(\rho v \otimes v)+\nabla P+\dfrac1\e \rho v=0, \end{matrix} \right.
\label{CED1}\end{equation} 
with $\varepsilon>0$ and for a pressure law $P$  satisfying
\begin{equation}\label{Pression1}P(\rho)=A\rho^\gamma\ \text{ for }\gamma>1\andf A>0.\end{equation}
Above, $v=v(t,x)\in\R^d$ designates the velocity field and
$\rho=\rho(t,x)>0$ the density of the fluid.
\medbreak
It is well known that  classical  systems of conservation laws (that is with $LV=0$) 
supplemented with 
smooth data  admit local-in-time strong solutions that may develop singularities (shock waves) in finite time even if the initial data are small  perturbations of a constant solution  (see for instance the works by Majda in \cite{Majda} and Serre in \cite{Serre}).  
A  sufficient condition of  global existence  for small perturbations of a constant
solution $\bar V$ of \eqref{GEQSYM} is 
the \emph{total dissipation hypothesis}, namely the  dissipation  term $LV$ acts directly on   each component of the system
(that is $L\!+\!{}^T\!L$ is positive definite), making
the whole solution  to tend to $\bar V$  exponentially fast, cf. the work of Li in \cite{LiTT}. 
\smallbreak
However, in most physical situations that may be  modelled by systems of the form \eqref{GEQSYM},
    some components of the solution  satisfy conservation laws 
    and only \emph{partial} dissipation occurs, that is to say, the term $LV$ acts only   on a part of the solution. 
   Typically, this happens in  gas dynamics  where the mass density and entropy are conserved,  or in numerical schemes  involving conservation laws with  relaxation.

 In his 1984 PhD thesis \cite{Kawa1},  Kawashima 
 found out a sufficient structure condition 
  on  a class of hyperbolic-parabolic  systems 
containing  \eqref{GEQSYM} 
  guaranteeing  the global existence of strong solutions for perturbations  of a constant state $\bar V$
 (see also \cite{SK}). 
 This criterion, now called the (SK) condition,
 has been revisited later by Yong in \cite{Yong}
 who noticed that the existence of a (dissipative) entropy which provides a suitable symmetrization of the system compatible with 0-th order term allows to get a global existence result for small data in $H^s$ spaces with $s>\frac{d}{2}+1.$
Later, by taking advantage of the properties of the Green kernel of the linearized system around $\bar V$ and on the Duhamel formula,  Bianchini, Hanouzet and Natalini in \cite{BHN} pointed  out the 
convergence of global solutions to $\bar V$ in $L^p,$    with the rate ${\mathcal O}(t^{-\frac{d}{2}(1-\frac{1}{p})})$ when $t\rightarrow\infty,$  for all $p\in[\min\{d,2\},\infty]$.   In \cite{KYDecay}, Kawashima and Yong proved 
decay estimates in  Sobolev spaces and, a few years ago,   Kawashima and Xu in \cite{XK1} obtained a global existence result for small data in critical   non-homogeneous Besov spaces.
 \medbreak
 In \cite{BZ},  Beauchard and Zuazua  developed a new and systematic approach  that allows to establish global existence  and to describe large time behavior   of solutions to   partially dissipative systems.
Looking at  the linearization of  System \eqref{GEQSYM} around a constant solution $\bar V$, namely
(denoting from now on $\d_t\triangleq \frac{\partial}{\partial t}$ and $\d_k\triangleq \frac{\partial}{\partial x_k}$),
\begin{equation} \label{LinearBZ}
\d_tZ+\sum_{k=1}^m\bar A^k\d_kZ=-\frac{LZ}\varepsilon\with \bar A^k\triangleq A^k(\bar V),
\end{equation}
they showed that the (SK) condition is equivalent to the Kalman maximal rank condition on the matrices $\bar A^k$ and $L.$
More importantly, they  introduced a Lyapunov functional 
equivalent to the $L^2$ norm 
 that encodes enough information to recover all the dissipative properties of \eqref{LinearBZ}.  
 Considering such a functional  is motivated by the classical (linear) control theory of ODEs, 
 and has some connections with  Villani's work on hypocoercivity  \cite{Villani}.
 Back to the nonlinear system \eqref{GEQSYM}, Beauchard and Zuazua  obtained  the existence
of global smooth solutions  for perturbations in $H^s$ 
with $s>d/2+1$ of a constant equilibrium $\bar V$
that satisfies Condition (SK).  Furthermore, using arguments borrowed from  Coron’s return method \cite{Coron},
they were able to  achieve  certain cases where (SK)  does not hold.
\medbreak
Using the techniques from \cite{BZ} and a deep analysis of the low frequencies of the solution, in \cite{CBD1,CBD2}, we investigated  the global existence 
issue for these systems in critical homogeneous Besov spaces first in the one-dimensional, and then in the multi-dimensional setting.

Our aim here is to extend these results to a  more general 
functional framework  where the  low frequencies of the solution are only bounded in some $L^p$-type space with $p>2.$ In this way, 
 one can take initial data that are less decaying at infinity
 than those that have been considered so far. 
 To compare with the prior works, one can keep in mind 
 the following chain of embedding that holds true for all $s>1+d/2$ and $p\geq2$:
$$H^{s}\hookrightarrow \B^{\frac{d}{2}+1}_{2,1}\hookrightarrow \dot{\B}^{\frac{d}{2}}_{2,1}\cap \dot{\B}^{\frac{d}{2}+1}_{2,1}\hookrightarrow \dot {\B}^{\frac{d}{p},\frac{d}{2}+1}_{p,2,1}\hookrightarrow \mathcal{C}^1_b.$$
The left space corresponds to the classical theory,  
generalized to the nonhomogeneous Besov space 
$\B^{\frac{d}{2}+1}_{2,1}$ by Kawashima and Xu in \cite{XK1}. 
The theory in $\dot{\B}^{\frac{d}{2}}_{2,1}\cap \dot{\B}^{\frac{d}{2}+1}_{2,1}$
has been performed by us in \cite{CBD2} while 
$\dot {\B}^{\frac{d}{p},\frac{d}{2}+1}_{p,2,1}$ -- the aim of the present paper -- 
amounts to assuming that the low (resp. high) frequencies of 
the data are in $\dot{\B}^{\frac{d}{p}}_{p,1}$ (resp. $\dot{\B}^{\frac{d}{2}+1}_{2,1}$).
\medbreak
Looking for global well-posedness results with low and high frequencies
of the data belonging to different type of Besov spaces
originates from the study of  the compressible Navier-Stokes system in \cite{NSCLP,CMZ}. There, the authors proved  a global well-posedness result with the high frequencies of the solution belonging to some homogeneous Besov space based on $L^p$  with $p\geq 2$ while the low frequencies are in a space related
to $L^2$. Essentially, this is possible because the  high frequencies' eigenvalues of the linearized system (in the Fourier variables) are  real-valued. 
In our case, the situation is reversed: the low frequencies' eigenvalues are  real
and can thus be handled in a $L^p$-type framework, 
while the high frequencies are complex-valued.
\medbreak
To achieve our results, we need to reconsider the way of treating the low frequencies compared to  \cite{CBD2}.
Similarly to what we did in \cite{CBD1}, we shall exhibit a damped mode with better decay properties than the whole solution and use it to diagonalize and decouple the system into a purely damped equation and a parabolic one. This, as well as a link between the (SK) condition and the ellipticity of a certain operator will be the keys to obtaining a priori estimates in the $L^p$ framework in low frequencies. To handle  the high frequencies, we shall work out a Lyapunov functional  in Beauchard-Zuazua's style, that is equivalent to the norm that we aim at controlling. We shall proceed as in \cite{CBD2}
but will take advantage of  new commutator and composition lemma to handle the low frequencies part of the nonlinearities that do not belong to $L^2$. This leads to a global  well-posedness result for  \eqref{GEQSYM}
(Theorem \ref{ThmExistLp}).
\medbreak
The advantage of using homogeneous norms is that
it allows us to keep  track of  the dependency of the estimates of the solution 
with respect to $\varepsilon$ by a mere rescaling, and to 
get uniform estimates in the asymptotics $\varepsilon\to0.$
This is a crucial step for investigating the  infinite relaxation limit
for general systems of type \eqref{GEQSYM}. 
As an application, we shall consider  the isentropic compressible Euler with damping, and justify its relaxation limit toward the porous media equation after performing the following so-called  diffusive rescaling:  \begin{equation}\label{DiffusiveRescaling}(\widetilde{\rho}^\varepsilon,\widetilde{v}^\varepsilon)(t,x)\triangleq (\rho,\varepsilon^{-1}v)(\varepsilon^{-1}t,x).\end{equation}
Then, System \eqref{CED1}  becomes
\begin{equation*}
\left\{
\begin{array}
[c]{l}%
\partial_{t}\rho^{\varepsilon}+\div(\rho^{\varepsilon} v^{\varepsilon})=0,\\[1ex]
\varepsilon^2 \partial_{t}(\rho^{\varepsilon}v^{\varepsilon})+\varepsilon^2\textrm{div} (\rho^{\varepsilon}v^{\varepsilon}\otimes v^{\varepsilon})+\nabla P(\rho^{\varepsilon})+\rho^{\varepsilon}v^{\varepsilon}=0.
\end{array}
\right.
\end{equation*}
As $\varepsilon\rightarrow0$,  we expect
$\rho^\varepsilon$
and  $u^{\varepsilon}$ to tend  to the solution of the following porous media equations:
 \begin{equation*}
 \left\{
 \begin{array}
 [c]{l}%
\partial_t\cN-\Delta P(\cN)=0,\\[1ex] 
 \cN v+\nabla P(\cN)=0,
 \end{array}
 \right.
 \end{equation*}
 where the second equation corresponds to Darcy's law. 
 \medbreak
The rigorous derivation of the  relaxation limit 
for systems of conservation laws
  can be tracked back to the work of Chen, Levermore and Liu in \cite{chen1994}. For the damped Euler equations, it was performed in \cite{mar0,CoulombelGoudon,mar1,hsiao1,Junca,XuWang,CoulombelLin} in various settings.

In \cite{Junca}, Junca and Rascle were able to justify the relaxation process from the damped Euler equations to the porous media equation in the one-dimensional setting for large global-in-time $BV$ solution and to provide an explicit convergence rate. Their approach was based on a stream function technique which is related to the Lagrangian mass coordinates. More recently, following the same approach as in \cite{Junca}, Peng et al. in \cite{Peng2019} justified the convergence of partially dissipative hyperbolic systems to parabolic systems globally-in-time in one space dimension and derived a convergence rate of the relaxation process. Using similar techniques, Liang and Shuai in \cite{LiangShuai} generalized the previous result to the multi-dimensional periodic setting (in the torus $\mathbb{T}^3$). As their result relies on the Poincaré inequality to control mixed partial derivative terms that do not appear in the one dimensional case, their method cannot handle global-in-space norms. 

In the context of smooth solutions, Coulombel, Goudon and Lin in \cite{CoulombelLin,CoulombelGoudon} proved the strong convergence locally-in-space in Sobolev spaces with regularity index $s>\frac{d}{2}+1.$ 
These works were extended by  Xu, Kawashima and Wang in \cite{XuWang,XK1E} 
to the setting of  non-homogeneous critical Besov spaces.
\smallbreak
In the present paper, our focus is on critical  solutions 
with regularity $\dot\B^{\frac dp}_{p,1}$ in low frequencies and 
$\dot\B^{\frac d2+1}_{2,1}$ in high frequencies. 
For the compressible Euler system \eqref{CED1}, 
we shall justify the \emph{strong} convergence to the porous media equation
with \emph{an explicit convergence rate}. This will be possible thanks to the uniform bounds from Theorem \ref{ThmExistLp} on the damped mode, a suitable threshold between the low and high frequencies depending on $\varepsilon$ and the diffusive rescaling \eqref{DiffusiveRescaling}. This result is given in Theorem \ref{Thm-relax}.
\medbreak
The rest paper is arranged as follows. In the first section, we specify the structure of the class  of partially dissipative hyperbolic systems we aim at considering and state our main results.
In the next section, we  give some insight on the strategies of the proofs.   
Section  \ref{s:GWP}   is devoted to the proof of our global existence
result for a class of partially dissipative systems satisfying the Shizuta-Kawashima condition. 
 In section \ref{s:lim}, we justify rigorously the relaxation limit of the compressible Euler system with damping to the porous media equation and derive an explicit convergence rate of the process.
 Some technical results are proved or recalled in Appendix.

\medbreak\noindent{\bf Acknowledgments.}  
 The authors are partially  supported  by the ANR project INFAMIE (ANR-15-CE40-0011). 


\section{Hypotheses and results} \label{s:hyp}

Before presenting our main results, 
motivating the structure assumptions that will 
be made on System \eqref{GEQSYM} is in order.
To this end, fix some constant reference solution 
 $\bar{V}$ (hence satisfying $L\bar{V}=0$). Setting $Z\triangleq V-\bar V,$ System \eqref{GEQSYM} becomes
  \begin{equation}\label{GE}
  \d_tZ+\sum_{k=1}^d  A^k(V)\d_kZ+\frac{LZ}{\varepsilon}=0.\cdotp\end{equation}
We assume that: 
\begin{enumerate}
\item[(i)] The range and kernel of operator  $L$ satisfy   
$$\mathbb{R}^n={\rm Ker}(L) \overset{\perp}{\oplus} {\rm Im}(L). $$
In what follows, we shall assume  with no loss of generality that 
$\ker L=\R^{n_1}\times\{0\}$ and  ${\rm Im}(L)=\{0\}\times\R^{n_2}$
with $n_1+n_2=n.$
\smallbreak\item[(ii)]  The restriction  $L_2$ of $L$
to ${\rm Im}(L)$ (identified to $\R^{n_2}$)
satisfies, for some $c>0,$ the positivity condition: 
  \begin{equation}\label{partdissip2}
\forall V_2\in \R^{n_2},\; (L_2V_2|V_2)\geq c|V_2|^2.
\end{equation}
\end{enumerate}
  \smallbreak
 Note that,  using the decompositions,
$$ V=\left(\begin{matrix}V_1 \\ V_2 \end{matrix}\right),\qquad
 LV=\left(\begin{matrix} 0\\L_2V_2 \end{matrix}\right) \andf A^k=\left(\begin{matrix}  A^k_{1,1} &  A^k_{1,2}\\ A^k_{2,1} &  A^k_{2,2} \end{matrix}\right),$$
 System \eqref{GE} can be rewritten:
\begin{equation} \left\{ \begin{matrix}\displaystyle \partial_tZ_1 + \sum_{k=1}^d\left(A_{1,1}^k(V)\partial_{k}Z_1+A_{1,2}^k(V)\partial_{k}Z_2\right)=0,\\ \displaystyle \partial_tZ_2 + \sum_{k=1}^d\left(A_{2,1}^k(V)\partial_{k}Z_1+A_{2,2}^k(V)\partial_{k}Z_2\right)+\frac{L_2Z_2}{\varepsilon}=0.\end{matrix} \right. \label{GE2}
\end{equation} 
Recall that we assumed from the very beginning that:
\begin{equation}\label{AkLinear}
\hbox{all the maps }\ Z\to A^k(\bar V+Z)\ \hbox{ are linear}.
\end{equation}
In order to obtain our results,  we shall additionally assume that
for all $k\in\{1,\cdots,d\},$
\begin{equation}\label{StructAssum}
\quad A_{1,1}^k(\bar V)=0 \andf  Z\mapsto A^k_{1,1}(\bar V+Z)\quad\hbox{is linear with respect to } \ Z_2.
\end{equation}
\begin{Rmq} The condition of linearity for $A^k(V)$ 
may be just technical. However, in our functional framework, 
we do not know how to handle more than quadratic nonlinearities when performing estimates in the high frequencies regime. 

At the same time, in contrast with \cite[Th. 2.3]{CBD2}, we do not have 
to assume here that $A^k_{1,2}$ and $A^k_{2,1}$ are linear with respect to $Z_2.$ This is due to a better treatment of the corresponding terms
(use of suitable commutator estimates recalled in Appendix). 
\end{Rmq}

In order to explain the  supplementary conditions 
that we need in order to get our main results, 
 let us consider  the linearization  of \eqref{GE} about $\bar V,$ namely,
 in the case $\varepsilon=1$:
   \begin{equation}\label{eq:Zlinear} \d_tZ+\sum_{k=1}^d\bar A^k \d_kZ + LZ=0 \with  \bar A^k:=A^k(\bar V)\ \hbox{ for }\ k=1,\cdots,d.\end{equation}
  Denoting by $\xi\in\R^d$ the Fourier variable corresponding 
  to $x\in\R^d$, we have
  $$\d_t\wh Z+i\sum_{k=1}^d\bar A^k \xi_k\wh Z + L\wh Z=0.$$
  Setting  $\xi= \rho \omega$ with $\omega\in\S^{d-1}$ and $\rho=|\xi|,$ the above system rewrites
  \begin{equation}\label{eq:Zomega}
  \d_t\wh Z +i\rho M_\omega \wh Z +L\wh Z = 0\with 
  M_\omega\triangleq \sum_{k=1}^d \omega_k\bar A^k.
  \end{equation}
  Clearly, \eqref{partdissip2} implies that  
   \begin{equation}\label{partdissip3}    (L\wh Z|\wh Z)_{L^2} \geq c\|\wh Z_2\|_{L^2}^2\quad\hbox{for all }
 \ Z\in L^2(\R^d;\R^n).\end{equation}
 This provides dissipation on the directly damped component $Z_2$.
 In order to have dissipation for \emph{all} the components of the solution 
 $Z$ to \eqref{eq:Zlinear}, Shizuta and Kawashima \cite{Kawa1,SK}
 proposed the following:
\begin{Def}
System \eqref{GEQSYM} verifies the (SK) condition  at $\bar{V}$ if, for all $\omega\in\S^{d-1},$ we have  at the same time  $L\phi=0$  and 
 $\lambda  \phi+M_\omega\phi=0$ for some   $\lambda\in\mathbb{R},$ 
 if and only if  $\phi=0_{\R^n}$. 
\end{Def}

 As we shall see in the last section of the paper, the compressible Euler system with damping, rewritten
 in suitable variables, satisfies Condition (SK) about any constant state with positive density and null velocity. 
\medbreak
Condition (SK) turns out to be equivalent 
to the strong stability of System \eqref{eq:Zlinear} and it
has been shown in \cite{SK} that if it is satisfied, then solutions to \eqref{eq:Zlinear} satisfy for all $t\geq0,$
\begin{equation}\begin{aligned} \label{Intro:Decomphfbf}
    \|Z^h(t)\|_{L^2(\R^d,\R^n)}&\leq Ce^{-\lambda t}\|Z_0\|_{L^2(\R^d,\R^n)}, 
    \\ \|Z^\ell(t)\|_{L^\infty(\R^d,\R^n)}&\leq Ct^{-\frac{d}{2}}\|Z_0\|_{L^1(\R^d,\R^n)},
\end{aligned}\end{equation}
where $Z^h$ and $Z^\ell$ correspond, respectively, to the high and low frequencies of the solution. Thus, at the linear level, we observe that, 
provided (SK) is satisfied:
\begin{itemize}
    \item in high frequencies, the solution decays exponentially to $0$;
    \item in low frequencies, the solution behaves as the solutions of the heat equation.
    \end{itemize}
Even though Condition (SK) ensures strong linear stability, 
it does not tell us how to achieve the corresponding quantitative estimates.
In \cite{BZ}, Beauchard and Zuazua proposed a new approach 
to build an explicit Lyapunov functional that allows to control 
suitable norms of the solution.  We took advantage of it in \cite{CBD2}
to  prove global existence
and decay estimates for solutions to \eqref{GEQSYM} in a critical $L^2$-type
framework. 
In the present paper, we will combine Beauchard and Zuazua's strategy 
to control the high frequencies of the solution, with 
estimates for a suitable `damped mode' so as to achieve $L^p$-type
estimates for the low frequencies of the solution.


At this stage, introducing  a few notations is in order. 
First,  we fix a homogeneous  Littlewood-Paley decomposition $(\ddj)_{j\in\Z}$
that is defined  by 
$$\ddj\triangleq\varphi(2^{-j}D)\with \varphi(\xi)\triangleq \chi(\xi/2)-\chi(\xi)$$
where $\chi$ stands for a  smooth function  with range in $[0,1],$ supported in the open ball $B(0,4/3)$ and
such that $\chi\equiv1$ on the closed ball $\bar{B}(0,3/4).$ 
We further state:
$$\dot S_j\triangleq \chi(2^{-j}D) \quad\hbox{for all }\ j\in\Z$$
and define  $\mathcal{S}'_h$ to be the set of tempered distributions $z$  such 
that $$\lim_{j\to-\infty}\|\dot S_jz\|_{L^\infty}=0.$$ 
Following  \cite{HJR}, we introduce the  homogeneous Besov semi-norms:
$$\|z\|_{\dot\B^s_{p,r}}\triangleq \bigl\| 2^{js}\|\ddj z\|_{L^p(\R^d)}\bigr\|_{\ell^r(\Z)},$$
then  define the homogeneous Besov spaces $\dot\B^s_{p,r}$ (for any $s\in\R$ and $(p,r)\in[1,\infty]^2$)
 to be the subset of $z$ in  $\mathcal{S}'_h$ such that $\|z\|_{\dot\B^s_{p,r}}$ is finite. 
\smallbreak
Use from now on the shorthand notation 
\begin{equation}\label{eq:zj}
\ddj z\triangleq z_j.
\end{equation}
For any fixed threshold $J\in\Z,$
we define the low and high frequency parts of any   $z\in\mathcal{S}'_h$ to be:
$$z^{\ell,J}:=\sum_{j\leq J-1}z_j \andf z^{h,J}:=\sum_{j\geq J}z_j. $$
Likewise, we  set\footnote{For technical reasons, we need
a small overlap between low and high frequencies.} if $r<\infty,$
$$ \norme{z}^{\ell,J}_{\dot{\mathbb{B}}^{s}_{p,r}}\triangleq \biggl(\sum_{j\leq J+1}\bigl(2^{js}\|\dot{\Delta}_jz\|_{L^p}\bigr)^r\biggr)^{\frac{1}{r}} \andf
\norme{z}^{h,J}_{\dot{\mathbb{B}}^{s}_{p,r}}\triangleq \biggl(\sum_{j\geq J}\bigl(2^{js}\|\dot{\Delta}_jz\|_{L^p}\bigr)^r\biggr)^{\frac{1}{r}}\cdotp$$ 
Whenever the value of $J$ is clear from the context,  it is omitted in the notations.
\medbreak
For any Banach space $X,$ index $\rho$ in $[1,\infty]$  and time $T\in[0,\infty],$ we use the notation 
$\|z\|_{L_T^\rho(X)}\triangleq  \bigl\| \|z\|_{X}\bigr\|_{L^\rho(0,T)}.$
If $T=\infty$, then we  just  write $\|z\|_{L^\rho(X)}.$
Finally, in the case where $z$ has $n$ components $z_k$ in $X,$ we 
slightly abusively keep the notation  $\norme{z}_X$ to 
mean $\sum_{k\in\{1,\cdots,n\}} \norme{z_k}_X$. 

\bigbreak
The first part of the paper is devoted to proving the following 
 global existence result with uniform estimates with respect 
 to the relaxation parameter for System~\eqref{GE}.
\begin{Thm} \label{ThmExistLp} Let $p\in[2,4]$ if $1\leq d\leq 4,$
or  $p\in[2,\frac{2d}{d-2}]$ if $d\geq5.$  Assume that the (SK) condition and  \eqref{AkLinear}-\eqref{StructAssum} are satisfied. There exist $k_p\in\Z$ and $c_0>0$ such that for 
all $\varepsilon>0,$
 if we assume that
$Z_0^{\ell,J_\varepsilon}\in{\dot{\mathbb{B}}^{\frac{d}{p}}_{p,1}}$ and $Z_0^{h,J_\varepsilon}\in{\dot{\mathbb{B}}^{\frac{d}{2}+1}_{2,1}}$ with
$$ \norme{Z_0}^{\ell,J_\varepsilon}_{\dot{\mathbb{B}}^{\frac{d}{p}}_{p,1}}+ \varepsilon\norme{Z_0}^{h,J_\varepsilon}_{\dot{\mathbb{B}}^{\frac{d}{2}+1}_{2,1}} \leq c_0\andf J_\varepsilon\triangleq\left\lfloor\rm -log_2\,\varepsilon\right\rfloor+k_p,$$
then System \eqref{GE} admits a unique global solution $Z$ in the space $E_p^{J_\varepsilon}$ defined by 
\begin{eqnarray*}
&&Z_1^{\ell,J_\varepsilon}\in \mathcal{C}_b(\R^+;\dot{\mathbb{B}}^{\frac{d}{p}}_{p,1})\cap L^1(\mathbb{R}^+;\dot{\mathbb{B}}^{\frac{d}{p}+2}_{p,1}), \;\;\;\; Z^{h,J_\varepsilon}\in \mathcal{C}_b(\R^+;\dot{\mathbb{B}}^{\frac{d}{2}+1}_{2,1})\cap L^1(\mathbb{R}^+,\dot{\mathbb{B}}^{\frac{d}{2}+1}_{2,1}), \\ &&
Z_2^{\ell,J_\varepsilon}\in \mathcal{C}_b(\R^+;\dot{\mathbb{B}}^{\frac{d}{p}}_{p,1})\cap L^1(\mathbb{R}^+;\dot{\mathbb{B}}^{\frac{d}{p}+1}_{p,1}), \:\:\: 
W_\varepsilon \in  L^1(\mathbb{R}^+;\dot{\mathbb{B}}^{\frac{d}{p}}_{p,1}) \andf Z_2\in L^2(\mathbb{R}^+;\dot{\mathbb{B}}^{\frac{d}{p}}_{p,1})\\
&&\with W_\varepsilon\triangleq \dfrac{Z_2}{\varepsilon}+\sum_{k=1}^d L_2^{-1}\bigl(A^k_{2,1}(V)\d_kZ_1+A_{2,2}^k(V)\d_kZ_2\bigr)\cdotp
\end{eqnarray*}
Moreover, we have the following  bound:
$$X_{p,\varepsilon}(t)\lesssim\norme{Z_0}^{\ell,J_\varepsilon}_{\dot{\mathbb{B}}^{\frac{d}{p}}_{p,1}}+ \varepsilon\norme{Z_0}^{h,J_\varepsilon}_{\dot{\mathbb{B}}^{\frac{d}{2}+1}_{2,1}}
\quad\hbox{for all } t\geq 0,$$  where 
$$\displaylines{
X_{p,\varepsilon}(t)\triangleq\norme{Z}^{\ell,J_\varepsilon}_{L^\infty_t(\dot{\mathbb{B}}^{\frac{d}{p}}_{p,1})}+\varepsilon\norme{Z}^{h,J_\varepsilon}_{L^\infty_t(\dot{\mathbb{B}}^{\frac{d}{2}+1}_{2,1})}
+\varepsilon\norme{Z_1}^{\ell,J_\varepsilon}_{L^1_t(\dot{\mathbb{B}}^{\frac{1}{p}+2}_{p,1})}+\norme{Z}^{h,J_\varepsilon}_{L^1_t(\dot{\mathbb{B}}^{\frac{d}{2}+1}_{2,1})}+\norme{Z_2}^\ell_{L^1_t(\dot{\mathbb{B}}^{\frac{d}{p}+1}_{p,1})}
\hfill\cr\hfill
+\norme{W_\varepsilon}_{L^1_t(\dot{\mathbb{B}}^{\frac{d}{p}}_{p,1})}+\varepsilon^{-\frac12}\norme{Z_2}_{L^2_t(\dot{\mathbb{B}}^{\frac{d}{p}}_{p,1})}.}$$ 
\end{Thm}
\begin{Rmq}
The choice of the threshold $J_\varepsilon$ 
corresponds\footnote{The value of  $k_p$ is given by our low frequencies analysis. At some point, we need the threshold to be small enough in 
order to close the estimates. As pointed out in \cite{CBD2}, for $p=2,$
one can take $k_p=0.$} to the place where the $0$-order terms and the 1-order terms have the same strength (parameter included). It can be easily deduced from a spectral analysis of the linearized system.
\end{Rmq}
The above theorem can  be applied to 
 the isentropic compressible Euler equation with damping
 \eqref{CED1} after suitable symmetrization. 
Indeed, introduce the following  `sound speed': 
\begin{equation}\label{eq:c}  c\triangleq\frac{2}{\gamma-1}\sqrt{\frac{\d P}{\d\rho}}=\frac{(4\gamma A)^\frac{1}{2}}{\gamma-1}\rho^{\frac{\gamma-1}{2}}.\end{equation}
 Then, setting $\check{\gamma}=\dfrac{\gamma-1}{2}$ and $\wt c=c-\bar{c}$,
we can rewrite \eqref{CED1}   under the form~:
\begin{equation} \left\{ \begin{aligned} &\partial_t\wt c+v\cdot\nabla \wt c+\check{\gamma}(\wt c+\bar{c})\textrm{div}\,v=0,\\ 
&\partial_tv+v\cdot\nabla v+\check{\gamma}(\wt c+\bar{c})\nabla \wt c+\frac1\e v=0. \end{aligned} \right.\label{CED4}
\end{equation} 
The above system is symmetric and satisfies
both Conditions (SK) and \eqref{AkLinear} - \eqref{StructAssum}.
Then, applying Theorem \ref{ThmExistLp} yields:
\begin{Thm} \label{ThmExistLpCED} 
Fix some positive constant density $\bar\rho.$ 
Let $p\in[2,4]$ if $1\leq d\leq 4,$
or  $p\in[2,\frac{2d}{d-2}]$ if $d\geq5.$ 
There exist $k_p\in\Z$ and $c_0>0$ such that for
all $\varepsilon>0,$ if we set
$J_\varepsilon\triangleq\left\lfloor\rm -log_2\varepsilon\right\rfloor+k_p$ and  assume that
$ (c-\bar{c})^{\ell,J_\varepsilon},v_0^{\ell,J_\varepsilon}\in{\dot{\mathbb{B}}^{\frac{d}{p}}_{p,1}}$ and 
$ (c-\bar{c})^{h,J_\varepsilon},v_0^{h,J_\varepsilon}\in{\dot{\mathbb{B}}^{\frac{d}{2}+1}_{2,1}}$ with

$$ \norme{(  c-\bar{c},v_0)}^{\ell,J_\varepsilon}_{\dot{\mathbb{B}}^{\frac{d}{p}}_{p,1}}+ \varepsilon\norme{(  c-\bar{c},v_0)}^{h,J_\varepsilon}_{\dot{\mathbb{B}}^{\frac{d}{2}+1}_{2,1}} \leq c_0,
$$
then System \eqref{CED4} admits a unique global solution 
$(c,v)$ with
$( c-\bar{c},v)$ in the space $E_p^{J_\varepsilon}$ defined by 
\begin{eqnarray*}
&&( c-\bar{c})^{\ell,J_\varepsilon}\in \mathcal{C}_b(\R^+;\dot{\mathbb{B}}^{\frac{d}{p}}_{p,1})\cap L^1(\mathbb{R}^+;\dot{\mathbb{B}}^{\frac{d}{p}+2}_{p,1}), \;\;\;\;  ( c-\bar{c})^{h,J_\varepsilon}\in \mathcal{C}_b(\R^+;\dot{\mathbb{B}}^{\frac{d}{2}+1}_{2,1})\cap L^1(\mathbb{R}^+;\dot{\mathbb{B}}^{\frac{d}{2}+1}_{2,1}), \\ &&
v^{\ell,J_\varepsilon}\in \mathcal{C}_b(\R^+;\dot{\mathbb{B}}^{\frac{d}{p}}_{p,1})\cap L^1(\mathbb{R}^+;\dot{\mathbb{B}}^{\frac{d}{p}+1}_{p,1}), \;\;\;\; v^{h,J_\varepsilon}\in \mathcal{C}_b(\R^+;\dot{\mathbb{B}}^{\frac{d}{2}+1}_{2,1})\cap L^1(\mathbb{R}^+;\dot{\mathbb{B}}^{\frac{d}{2}+1}_{2,1}),  \\&& \frac{v}{\varepsilon}+\check{\gamma}c\nabla c \in  L^1(\mathbb{R}^+;\dot{\mathbb{B}}^{\frac{d}{p}}_{p,1})\andf v\in L^2(\mathbb{R}^+;\dot{\mathbb{B}}^{\frac{d}{p}}_{p,1}).
\end{eqnarray*}
Moreover, we have the following a priori bound:
$$X_{p,\varepsilon}(t)\lesssim\norme{( c-\bar{c},v_0)}^{\ell,J_\varepsilon}_{\dot{\mathbb{B}}^{\frac{d}{p}}_{p,1}}+ \varepsilon\norme{( c-\bar{c},v_0)}^{h,J_\varepsilon}_{\dot{\mathbb{B}}^{\frac{d}{2}+1}_{2,1}}
\quad\hbox{for all } t\geq 0,$$  where 
$$\displaylines{
X_{p,\varepsilon}(t)\triangleq\norme{(  c-\bar{c},v)}^{\ell,J_\varepsilon}_{L^\infty_t(\dot{\mathbb{B}}^{\frac{d}{p}}_{p,1})}+\varepsilon\norme{(  c-\bar{c},v)}^{h,J_\varepsilon}_{L^\infty_t(\dot{\mathbb{B}}^{\frac{d}{2}+1}_{2,1})}
+\varepsilon\norme{ c-\bar{c}}^{\ell,J_\varepsilon}_{L^1_t(\dot{\mathbb{B}}^{\frac{d}{p}+2}_{p,1})}+\norme{v}^\ell_{L^1_t(\dot{\mathbb{B}}^{\frac{d}{p}+1}_{p,1})}\hfill\cr\hfill+\norme{(  c-\bar{c},v)}^{h,J_\varepsilon}_{L^1_t(\dot{\mathbb{B}}^{\frac{d}{2}+1}_{2,1})}
+\norme{\frac{v}{\varepsilon}+\check{\gamma}c\nabla c}_{L^1_t(\dot{\mathbb{B}}^{\frac{d}{p}}_{p,1})}+\varepsilon^{-\frac12}\norme{v}_{L^2_t(\dot{\mathbb{B}}^{\frac{d}{p}}_{p,1})}.}$$
\end{Thm}

\begin{Rmq}
According to  Theorem \ref{ThmExistLp}, the damped mode should be
 $\displaystyle W_\varepsilon\triangleq v\cdot \nabla v+\frac{v}{\varepsilon}+\check{\gamma}c\nabla c.$
 However, the above estimate combined with product laws
 ensure that  $\|v\cdot\nabla v\|_{L^1_t(\dot{\mathbb{B}}^{\frac{d}{p}}_{p,1})}\lesssim c_0^2,$ hence
 is negligible compared to $\varepsilon^{-1}v.$
Consequently,   $W'_\varepsilon\triangleq\dfrac{v}{\varepsilon}+\check{\gamma}c\nabla c$  can be seen as a damped quantity. This latter function turns out to be more adapted to the study of  the relaxation limit.
\end{Rmq}

The uniform estimates from Theorem \ref{ThmExistLpCED} enable us to obtain the following result pertaining to the relaxation limit of the compressible Euler system.

\begin{Thm} \label{Thm-relax}
Let the hypotheses of Theorem \ref{ThmExistLpCED} be in force
and denote by $(c,v)$ the corresponding solution.
Let $\rho$ be the density corresponding to $c$ through relation \eqref{eq:c}. 

Let the positive function $\mathcal{N}_0$
be such that $\mathcal{N}_0-\bar\rho$ 
is small enough in $\dot \B^{\frac{d}{p}}_{p,1},$ and 
let $\mathcal{N}\in\cC_b(\R^+;\dot\B^{\frac dp}_{p,1})\cap 
L^1(\R^+;\dot\B^{\frac dp+2}_{p,1})$ be the  global solution of: 
\begin{equation} \partial_t\mathcal{N}-\Delta P(\mathcal{N})=0 \label{M-eq}
\end{equation}
supplemented with initial data $\mathcal{N}_0$ given by Proposition \ref{PropExistMAppendix}.
\medbreak
Let $(\widetilde{\rho}^\varepsilon,\widetilde{v}^\varepsilon)(t,x)\triangleq (\rho,\varepsilon^{-1}v)(\varepsilon^{-1}t,x)$
and assume that $$\|\widetilde{\rho}_0^\varepsilon-\mathcal{N}_0\|_{\B^{\frac dp-1}_{p,1}}\leq C\varepsilon.$$
Then, as $\varepsilon\to0$, we have $$\widetilde{\rho}^\varepsilon- \mathcal{N} \longrightarrow 0 \quad \text{strongly in} \quad 
L^\infty(\R^+;\dot \B^{\frac{d}{p}-1}_{p,1})\cap L^1(\R^+;\dot \B^{\frac{d}{p}+1}_{p,1}),$$ and 
$$\widetilde{v}^\varepsilon+\frac{\nabla P(\widetilde{\rho}^\varepsilon)}{\widetilde{\rho}^\varepsilon}\longrightarrow 0 \quad \text{strongly in} \quad L^1(\R^+;\dot{\mathbb{B}}^{\frac{d}{p}}_{p,1}).  $$
Moreover, we have the following quantitative estimate:  
$$\|\widetilde{\rho}^\varepsilon-\mathcal{N}\|_{L^\infty(\R^+;\dot \B^{\frac{d}{p}-1}_{p,1})}+\|\widetilde{\rho}^\varepsilon-\mathcal{N}\|_{L^1(\R^+;\dot \B^{\frac{d}{p}+1}_{p,1})}+\norme{\frac{\nabla P(\widetilde{\rho}^\varepsilon)}{\widetilde{\rho}^\varepsilon}+\widetilde v^\varepsilon}_{L^1(\R^+;\dot{\mathbb{B}}^{\frac{d}{p}}_{p,1})}\leq C\varepsilon.$$
\end{Thm}

\begin{Rmq} 
To the best of our knowledge, all the previous works devoted to the relaxation limit
were based on compactness methods, which prevents to get strong convergence results
\emph{globally in space and time and explicit convergence rates} 
(see e.g. \cite{CoulombelLin,WasiolekPeng,XuWang}). Here our functional framework  allows us to bound directly  some   
norms  of the difference of the solutions.
Having  estimates on the damped mode at hand plays a key role.
\end{Rmq}
\begin{Rmq} In the specific case $p=2,$ we expect to have similar results 
for the compressible Euler equations supplemented with 
any  pressure law $P$ 
satisfying  
\begin{equation}\label{Pression2}
P'(\bar\rho)>0\quad\hbox{at some }\ \bar\rho>0.\end{equation}
Indeed, it is then possible to use the symmetrization 
of \cite{CBD2} (instead of the sound speed, 
we consider the unknown $n$ defined from $\rho$ through the relation:
 $ n(\rho)=\int_{\bar{\rho}}^\rho \frac{P'(s)}{s}\,ds$)
 and to use the standard composition lemma to treat general 
nonlinearities that need not be quadratic.  
\end{Rmq}


\section{A few words on the proofs} \label{strategy}

As a first, let us observe that performing 
a suitable time and space rescaling
(see \eqref{eq:rescaling} below) reduces the proof of global 
well-posedness for \eqref{GEQSYM} to the case $\varepsilon=1.$
Then, it is essentially a matter of establishing global-in-time a priori estimates for smooth 
solutions (from them and rather classical arguments, one can obtain the existence of global solutions). 
The uniqueness part of Theorem \ref{ThmExistLp} follows from stability 
estimates. As in \cite{CBD1}, the norms that will be used
in the stability estimates are not the standard ones owing 
to the use of the $L^p$ framework for the low frequencies.
\medbreak
In what follows,  we shortly explain how we proceed to prove 
global a priori estimates, 
adopting, as pointed out in the introduction, 
 a different strategy to handle the low and the high frequencies, 
 then we explain how to investigate the infinite relaxation limit. 

\subsection{Low frequencies: the damped mode}

Although Kawashima's decomposition \eqref{Intro:Decomphfbf} gives the overall behavior of the solution of the linearized system, 
it does not  fully reflect that a part of the 
solution has better decay properties in the low frequencies regime.
The essential ingredient of our low frequencies analysis is
to look at the time evolution of a `damped mode'  corresponding  to the part of the solution that experiences
maximal dissipation in low frequencies. 
In the case $\varepsilon=1$, it may be defined as follows~:
\begin{equation}\label{def:WW}
W\triangleq -L_2^{-1}\d_tZ_2=Z_2+\sum_{k=1}^d L_2^{-1}\bigl(A^k_{2,1}(V)\d_kZ_1+A_{2,2}^k(V)\d_kZ_2\bigr)\cdotp
\end{equation}
Hence $W$ satisfies
\begin{equation}\label{eq:W}\d_tW+L_2W =L_2^{-1}\sum_{k=1}^d\d_t\bigl(A^k_{2,1}(V)\d_{k}Z_1+A_{2,2}^k(V)\d_{k}Z_2\bigr)\cdotp\end{equation}
On the left-hand side,  Property \eqref{partdissip2} ensures maximal dissipation on $W.$
As  the right-hand side of \eqref{eq:W} contains only 
at least quadratic terms, or linear terms with  one derivative, it can be expected to be  negligible in low frequencies. Furthermore, the second equality of \eqref{def:WW} reveals that  $W$ is comparable to $Z_2$ 
in low frequencies if $Z$ is small enough. 
This will allow us to recover better integrability properties on  $Z_2$ than for the whole solution $Z.$ 

In  order to get optimal information, it is more suitable
to look at the evolution equations of $Z_1$ and $W,$
and we thus have to rewrite the equation of $Z_1$ in terms of
$W$ instead of $Z_2.$ 
Let us denote for all $k\in\{1,\cdots,d\}$
and $(p,m)\in\{1,2\}^2,$
$$\bar A^k_{p,m}\triangleq A^k_{p,m}(\bar V)\andf
\wt A^k_{p,m}(Z)\triangleq \wt A^k_{p,m}(\bar V+Z)-
\bar A^k_{p,m}.$$
Then, we rewrite 
the equation of $Z_1$
as follows: 
\begin{equation}\label{Z1withW}\partial_tZ_1 + \sum_{k=1}^d\bar A_{1,1}^k\partial_{k}Z_1-\sum_{k=1}^d\sum_{\ell=1}^d \bar A_{1,2}^kL_2^{-1}\bar A^{\ell}_{2,1}\d_k\d_\ell Z_1=f_1+f_2+f_3+f_4+f_5,
\end{equation}
where   
$$ 
 \begin{aligned} f_1&=\sum_{k=1}^d\sum_{\ell=1}^dA_{1,2}^k(V)\d_k\bigl(L_2^{-1}A_{2,2}^\ell(V)\d_\ell Z_2\bigr),\\
  f_2&=-\sum_{k=1}^dA_{1,2}^k(V)\partial_{k}W,\\ f_3&=\sum_{k=1}^d\sum_{\ell=1}^d A_{1,2}^k(V)\partial_k(L_2^{-1}\wt A^{\ell}_{2,1}(Z))\d_\ell Z_1),\\ 
 f_4&=\sum_{k=1}^d\sum_{\ell=1}^d \wt A_{1,2}^k(Z))L_2^{-1}\bar A^{\ell}_{2,1}\d_k\d_\ell Z_1,\\
  f_5&=-\sum_{k=1}^d \wt A_{1,1}^k(Z)\d_kZ_1.
\end{aligned}$$
As in \cite{WasiolekPeng}, in order to ensure the parabolic behavior of the unknown $Z_1$ in  low frequencies,  we 
make the following hypothesis: 
  \begin{equation}\label{eq:ellipticity}
 \forall k\in\{1,\cdots,d\},\;  \bar A_{1,1}^k=0\andf 
    \cA\triangleq -\sum_{k=1}^d\sum_{\ell=1}^d \bar A_{1,2}^kL_2^{-1}\bar A^{\ell}_{2,1}\d_k\d_\ell \text{ is strongly elliptic.}
\end{equation}
It turns out that if the first condition of \eqref{eq:ellipticity} is satisfied then 
the strong ellipticity of  $\cA$ is equivalent to Condition (SK)
(see the proof in Appendix). 
\begin{Rmq}
In the context of fluid mechanics, the first condition of \eqref{eq:ellipticity} 
is satisfied up to a Galilean change of frame.  Indeed,
$Z_1$ then corresponds to conserved quantities like e.g. the density or the entropy, and
$\sum_{k=1}^dA_{1,1}^k(V)\partial_kZ_1$ represents the corresponding transport term by the velocity field. 

Further remark that even if the matrices $\bar A_{1,1}^k$ are nonzero, the above sum  has no contribution in the energy arguments, and can be 
treated in a purely $L^2$ framework like in  \cite{CBD2}. 

Finally, in the general $L^p$ framework, our results still hold
if all the matrices $\bar A^k_{1,1}$ are diagonal. 
\end{Rmq}

Now, as \eqref{eq:ellipticity} is satisfied, one can take advantage of
parabolic maximal regularity estimates (see Lemma \ref{l:parabolic}) 
to bound $Z_1$ in Besov spaces of type $\dot\B^s_{p,1}.$
 Then, pointing out that the equations of $W$ and $Z_1$ are coupled by at least quadratic terms, or by linear terms of higher order (in terms of derivatives) 
that are thus negligible provided that the threshold $J_1$ between 
low and high frequencies is negative enough, 
one can establish the a priori estimates in the low frequencies region.

\subsection{High frequencies: Lyapunov functional}

We proceed as in \cite{CBD2}, but our nonstandard functional framework
regarding the low frequencies 
will complicate the treatment of some nonlinear terms. 

Recall that, in  Fourier variables, System \eqref{eq:Zlinear}  reads
  \begin{equation}\label{eq:Zomega2}
  \d_t\wh Z +i\rho M_\omega \wh Z +L\wh Z = 0\with 
  M_\omega\triangleq \sum_{k=1}^d \omega_k\bar A^k. 
  \end{equation}
To compensate the lack of coercivity depicted in the previous section, we  introduce a "correction" term to exhibit the time integrability properties for the un-damped components.
  Fix $n-1$ positive parameters $\e_1,\cdots \e_{n-1}$ (bound to be small) and set
  \begin{equation}\label{def:I}
  \cI\triangleq \Im \sum_{k=1}^{n-1} \e_k \bigl(LM_\omega^{k-1}\wh Z\cdot L M_\omega^k\wh Z\bigr)
  \end{equation}
  where $\cdot$ designates  the Hermitian scalar product in $\C^n.$ 
  \smallbreak
  In \cite{BZ}, the authors proved the following result

  \begin{Lemme}
  There exist positive parameters $\e_1,\cdots, \e_{n-1}$  so that 
  \begin{align}\label{eq:ZZ1}
\frac d{dt}\cL +\cH \leq 0\with \cH&\triangleq\int_{\R^d} \sum_{k=0}^{n-1} \e_k \min(1,|\xi|^2) |LM_\omega^k\wh Z(\xi)|^2\,d\xi \\\andf
\cL&\triangleq \|Z\|_{L^2}^2 + \int_{\R^d} \min(|\xi|,|\xi|^{-1})\cI(\xi)\,d\xi,\nonumber\end{align}
with, in addition, $\cL\simeq \|Z\|_{L^2}^2.$
  \end{Lemme}
  The question now is whether $\cH$ 
may be compared to   $\|Z\|_{L^2}^2.$ The answer depends on the properties of 
the support of $\wh Z_0$  and on the  possible cancellation of the following quantity:
\begin{equation}\label{eq:Nomega}
\mathcal{N}_{\bar{V}}:=\inf\biggl\{\sum_{k=0}^{n-1}\varepsilon_k|LM_\omega^kx|^2;\;x\in\S^{n-1},\,\omega\in\S^{d-1}\biggr\}\cdotp\end{equation}
In order to pursue our analysis, we  need the following key result (see the proof in e.g. \cite{BZ}). 
\begin{Prop}
Let $M$ and $L$ be two matrices in  $\cM_n(\R).$
The following assertions are equivalent:
\begin{enumerate}
\item $L\phi=0$ and $\lambda \phi+ M\phi=0$ for some $\lambda\in\R$
implies $\phi=0$;
\item For every $\e_0,\dotsm,\e_{n-1}>0$, the function 
$$y\longmapsto\sqrt{\sum_{k=0}^{n-1}\e_k|LM^ky|^2}$$
defines a norm on $\mathbb{\R}^n.$
\end{enumerate}
\end{Prop}

Thanks to the above proposition and observing that the unit sphere $\mathbb{S}^{d-1}$ is compact, one may   conclude that 
the (SK) condition  is satisfied by the pair $(M_\omega,L)$ for all $\omega\in\S^{d-1}$ if 
and only if $\cN_{\bar{V}}>0.$ Furthermore, we note that:
\begin{itemize}
\item if $\wh Z_0$ is compactly supported then $\cH\gtrsim \|\nabla Z\|_{L^2}^2$ (parabolic behavior of the low frequencies of the solution);
\item if the support of $\wh Z_0$ is away from the origin, then   $\cH\gtrsim \|Z\|_{L^2}^2$ (exponential decay for the high frequencies). 
\end{itemize}
We thus readily recover  Kawashima's decomposition \eqref{Intro:Decomphfbf}. 
\smallbreak 
As in \cite{CBD2}, an important part of the proof consists in  studying the evolution of the functional~$\cL.$ One cannot repeat exactly the  computations therein however, since the nonlinearities contain a little amount of low frequencies which are not bounded in $L^2$-based spaces.
Taking advantage of  appropriate product laws 
and commutator estimates  (see the Appendix) will enable us
to overcome the difficulty. 

\subsection{The relaxation limit} 

We finally give some  insight on our study of the limit $\varepsilon\to0$ 
for the isentropic compressible Euler equations with relaxation.  
At first sight, when the damping parameter $1/\varepsilon$ increases, we expect the dissipation  to  dominate more and more. 
This is not quite the case however, owing to the so-called \textit{overdamping} 
phenomenon: somehow, the `overall' damping rate behaves like $\min (\varepsilon,1/\varepsilon)$. This can be seen from a spectral analysis of the Euler system for example (see also \cite{SlideZuazua} for the case of the damped harmonic oscillator). 
Looking at the one-dimensional case for simplicity, the linearized Euler equations 
with relaxation parameter $\varepsilon$ has the following matrix in the Fourier side:
$$\begin{pmatrix}0 & i\xi \\ i\xi & {\varepsilon}^{-1} \end{pmatrix}\cdotp$$ 
It is straightforward that: 
\begin{itemize}
   \item for  frequencies $\xi$ such that $|\xi|\leq1/(2\varepsilon),$ 
   the above  matrix has two real eigenvalues asymptotically equal to ${1}/{\varepsilon}$ and $\varepsilon\xi^2,$ respectively  as $\xi$ tends to $0$;
\item for  frequencies $\xi$ such that $|\xi|\geq 1/(2\varepsilon),$ it has two complex conjugated eigenvalues with real part equal to  ${1}/{2\varepsilon}.$
\end{itemize}
It seems that in  the prior literature dedicated to the relaxation
limit, 
 the dissipative aspect of the low frequencies has been overlooked.
In fact, only  the overall behavior of the solution was taken 
into account and the threshold between the low and high frequencies was not taken  depending on $\varepsilon.$
 
In the approach that we offer here, we take advantage of  the uniform bounds on the damped mode which 
corresponds to the  eigenvalue asymptotically equal to $1/(2\varepsilon)$ 
for $\xi\to0,$ and allow  the threshold between low and high frequencies 
to depend on $\varepsilon$  by putting it at
 the place where the $0$-order terms and the $1$-order terms have the same weight (parameter included).
 This corresponds  approximately to $1/\varepsilon,$ in accordance with our spectral analysis above. 
 Somehow, as $\varepsilon\to0$, the low frequencies `invade' the whole space of frequencies.
 Hence the fact that  the limit system  for one of the components of the system 
 (namely the density)  has to be  purely parabolic does not come up as a surprise. 


With this heuristics in hand and
after performing  the diffusive rescaling \eqref{DiffusiveRescaling} and using the uniform bounds on the damped mode and on the solution from the uniform existence theorem, 
it is rather easy to  estimate  the difference between the solutions of the compressible Euler System and the porous media equation.
In this way, we get Theorem \ref{Thm-relax}.


\section{Global existence in the  \texorpdfstring{$L^p$}{TEXT} framework} \label{s:GWP} 

This section is devoted to proving  Theorem \ref{ThmExistLp}.  
We shall focus on the case $\varepsilon=1,$ which is not restrictive owing to the rescaling:
\begin{equation}\label{eq:rescaling} Z(t,x)\triangleq\check Z\Bigl(\frac{t}{\varepsilon},\frac{x}{\varepsilon}\Bigr)\cdotp\end{equation}
Indeed, $Z$ satisfies \eqref{GEQSYM} with relaxation parameter $\varepsilon$
if and only if $\check Z$ satisfies \eqref{GEQSYM} with relaxation parameter $1,$ and performing  the inverse scaling
gives us the desired dependency with respect to $\varepsilon$ in Theorem \ref{ThmExistLp}
since (see e.g. \cite[Chap.2]{HJR})  for all $s\in\R$ and $m\in[1,\infty],$ we have
\begin{equation}\label{eq:rescale}
\|z(\varepsilon\cdot)\|_{\dot\B^s_{m,1}}\simeq \varepsilon^{s-d/m}\|z\|_{\dot\B^{s}_{m,1}}
\end{equation}
and, by the same token, we have
\begin{equation}\label{eq:rescale2}
\|z(\varepsilon\cdot)\|^{\ell,J_1}_{\dot\B^s_{p,1}} \simeq \varepsilon^{s-d/p}\|z\|^{\ell,J_\varepsilon}_{\dot\B^s_{p,1}}\andf 
\|z(\varepsilon\cdot)\|^{h,J_1}_{\dot\B^s_{2,1}} \simeq \varepsilon^{s-d/2}\|z\|^{h,J_\varepsilon}_{\dot\B^s_{2,1}}.
\end{equation}
We shall use repeatedly that for $s\leq s',$ the following inequalities hold true:
\begin{equation}\label{eq:comparaison}
\|z^{\ell,J_1}\|_{\dot\B^{s'}_{2,1}}\lesssim \|z\|^{\ell,J_1}_{\dot\B^{s'}_{2,1}}\lesssim 2^{J_1(s'-s)}\|z\|^{\ell,J_1}_{\dot\B^{s}_{2,1}}\andf
\|z^{h,J_1}\|_{\dot\B^{s}_{2,1}}\lesssim\|z\|^{h,J_1}_{\dot\B^{s}_{2,1}}\lesssim 2^{J_1(s-s')}\|z\|^{h,J_1}_{\dot\B^{s'}_{2,1}}.\end{equation}

\subsection{A priori estimates}

Throughout this part  we set the threshold between low and high frequencies at some integer $J_1$ the value of which will be chosen during the computations. 
For better readability,  the exponent $J_1$ on the Besov norms will be omitted. 

We assume that we are given a smooth (and decaying) solution $Z$ 
of \eqref{GE} on $[0,T]\times\R^d$ with $Z_0$ as initial data, satisfying
\begin{equation}\label{eq:smallZ}
\sup_{t\in[0,T]} \|Z(t)\|^\ell_{\dot\B^{\frac dp}_{p,1}}+\sup_{t\in[0,T]} \|Z(t)\|^h_{\dot\B^{\frac d2+1}_{2,1}}\ll1.
\end{equation} 
We shall use repeatedly that, owing to 
\eqref{eq:comparaison} and to the embedding $\dot\B^{\frac dp}_{p,1}\hookrightarrow L^\infty$, we have
\begin{equation}\label{eq:smallZbis}
\sup_{t\in[0,T]} \|Z(t)\|_{L^\infty}\ll1
\end{equation} 
and also that for all $\sigma\in\R,$ $q\in[1,\infty]$ and $\alpha\geq0,$
\begin{equation}\label{eq:bernstein}
\|z\|_{\dot\B^{\sigma+\alpha}_{q,1}}^\ell
\lesssim \|z\|_{\dot\B^{\sigma}_{q,1}}^\ell\andf
\|z\|_{\dot\B^{\sigma-\alpha}_{q,1}}^h
\lesssim \|z\|_{\dot\B^{\sigma}_{q,1}}^h.
    \end{equation}

From now on, $C>0$ designates a generic harmless constant, the 
value of which depends on the context and we denote by $(c_j)_{j\in\Z}$
various nonnegative sequences such that  $\sum_{j\in\Z} c_j=1.$

\subsubsection{Low frequencies analysis in \texorpdfstring{$L^p$}{TEXT} spaces}
In this part, we fix the threshold $J_1$ between the low and high 
frequencies to be a negative enough integer.
We aim at proving:  
\begin{Prop} \label{APLP}
Assume that $2\leq p\leq4$ if $d\leq4,$ or $d\in[2,2d/(d-2)]$ if $d\geq5.$
There exists a positive real number $\kappa_0$ such that for all $t\in[0,T],$
\begin{multline}\label{eq:APLP}
\|(Z,W)(t)\|^\ell_{\dot\B^{\frac dp}_{p,1}}
  + \kappa_0\int_0^t\biggl(\|Z_1\|^\ell_{\dot\B^{\frac dp+2}_{p,1}}+\|Z_2\|^\ell_{\dot\B^{\frac dp+1}_{p,1}}+
\|W\|_{\dot{\mathbb{B}}^{\frac{d}{p}}_{p,1}}^\ell\biggr)\\
 \lesssim  \|(Z_{0},W_0)\|^\ell_{\dot\B^{\frac dp}_{p,1}}+\int_0^t\|Z\|_{\dot\B^{\frac dp}_{p,1}}\| Z_2\|_{\dot\B^{\frac dp+1}_{p,1}}+\int_0^t\|Z_2\|_{\dot\B^{\frac dp}_{p,1}}\|Z\|_{\dot\B^{\frac dp+1}_{p,1}}\\+\int_0^t\|Z\|_{\dot\B^{\frac dp+1}_{p,1}}^2+\int_0^t\|Z\|_{\dot\B^{\frac dp}_{p,1}}\|Z\|^h_{\dot\B^{\frac d2+1}_{2,1}}+\int_0^t\|W\|_{\dot\B^{\frac dp}_{p,1}}\|Z\|_{\dot\B^{\frac dp}_{p,1}\cap\dot\B^{\frac dp+1}_{p,1}}.\end{multline}
\end{Prop}
First, we state a result that will be used on several occasions.
\begin{Lemme} \label{dtZLp}
Under hypotheses \eqref{StructAssum} and \eqref{eq:smallZ}, we have  
 for all   $\sigma\in]d/p-d/p^*, d/p]$,
$$\begin{aligned}
\norme{\d_tZ_1}_{\dot{\mathbb{B}}^{\sigma}_{p,1}}&\lesssim 
\norme{\nabla Z_2}_{\dot{\mathbb{B}}^{\sigma}_{p,1}}+ \|Z_2\|_{\dot\B^{\frac dp}_{p,1}}
\|\nabla Z_1\|_{\dot\B^{\sigma}_{p,1}},\\
\norme{\d_tZ_2}_{\dot{\mathbb{B}}^{\sigma}_{p,1}}&\lesssim 
\norme{W}_{\dot{\mathbb{B}}^{\sigma}_{p,1}}.\end{aligned}$$
\end{Lemme}
\begin{proof}  
The second inequality just stems from the fact that 
the definition of $W$ in \eqref{def:WW}  is equivalent to
\begin{equation}\label{eq:WW}\d_tZ_2=-L_2 W.\end{equation}
The proof of the first item relies on  the explicit expression
of $\d_tZ_1$: since  $\bar A^k_{1,1}=0,$ we have for all $k\in\{1,\cdots,d\},$
\begin{equation}\label{eq:Z1}
\d_tZ_1 +\sum_{k=1}^d
\bar A^k_{1,2}\d_kZ_2= -
\sum_{k=1}^d\Bigl(\wt A^k_{1,1}(Z_2)\d_kZ_1 + \wt A^k_{1,2}(Z)\d_kZ_2\Bigr)\cdotp
\end{equation}
All the terms of the right-hand side are at least quadratic, and   \eqref{StructAssum}, \eqref{eq:smallZ}
thus ensure the desired inequality for $\d_tZ_1.$ 
\end{proof}
We now turn to the proof of Proposition \ref{APLP}.
\medbreak
\subsubsection*{Step 1: Estimate for  \texorpdfstring{$W$}{TEXT}.} We have the following statement.

\begin{Prop}\label{PropW} 
Denoting by $c$ the constant in \eqref{partdissip2}, we have
\begin{align}\label{est:W}
\|W(t)\|_{\dot{\mathbb{B}}^{\frac{d}{p}}_{p,1}}^\ell+ c\int_0^t \|W\|_{\dot{\mathbb{B}}^{\frac{d}{p}}_{p,1}}^\ell &\leq
\|W_0\|_{\dot{\mathbb{B}}^{\frac{d}{p}}_{p,1}}^\ell+ C\int_0^t  \|(\nabla Z_2,W)\|_{\dot\B^{\frac dp+1}_{p,1}}^\ell
\notag\\&+ C\int_0^t  \|(\nabla Z_2,W)\|_{\dot\B^{\frac dp}_{p,1}} \|Z\|_{\dot\B^{\frac dp}_{p,1}\cap\dot\B^{\frac dp+1}_{p,1}}
+ C\int_0^t  \|Z_2\|_{\dot\B^{\frac dp}_{p,1}}\|\nabla Z_1\|_{\dot\B^{\frac dp}_{p,1}}.\end{align}
\end{Prop}

\begin{proof}
From  \eqref{eq:W}, we gather that 
\begin{equation}\label{eq:eqW}
\d_tW+{L_2} W = h\triangleq {L_2}^{-1}h_1\\
\with  h_1\triangleq \sum_{k=1}^d\d_t\bigl(A^k_{2,1}(V)\d_kZ_1+A_{2,2}^k(V)\d_kZ_2\bigr)\cdotp
 \end{equation}
Applying  $\ddj$ to \eqref{eq:eqW} and taking the scalar product 
with $W_j|W_j|^{p-2}$ (where $W_j\triangleq \ddj W$) yields, thanks to \eqref{partdissip2}, 
\begin{equation}\label{eq:W0}\frac1p\frac d{dt}\|W_j\|_{L^p}^p +c\|W_j\|_{L^p}^p\leq C \|\ddj h_1\|_{L^p} \|W_j\|_{L^p}.
\end{equation}

 For bounding  $h_1,$  we use that
 for all $k\in\{1,\cdots,d\},$ 
 $$\displaylines{\d_t(A^k_{2,1}(V)\d_k Z_1 +A^k_{2,2}(V)\d_k Z_2) = D_VA^k_{2,1}(V)\d_t Z\d_kZ_1 
 +\bar A^k_{2,1}\d_t\d_kZ_1  \hfill\cr\hfill+  \wt A^k_{2,1}(Z)\d_t\d_kZ_1
+ D_VA^k_{2,2}(V)\d_t Z\d_kZ_2 
 +\bar A^k_{2,2}\d_t\d_kZ_2+  \wt A^k_{2,2}(Z)\d_t\d_kZ_2.}$$
 For $m=1,2,$ we have, according to Proposition \ref{LP}, Lemma \ref{dtZLp}
 and the fact that $D_VA^k_{2,m}$ for $m=1,2$ is constant,
  $$\begin{aligned} \|D_VA^k_{2,m}(V)\d_t Z\d_kZ_m\|_{\dot\B^{\frac{d}{p}}_{p,1}}^\ell&\lesssim 
 \|\d_t Z\|_{\dot\B^{\frac dp}_{p,1}} \|\nabla Z\|_{\dot\B^{\frac{d}{p}}_{p,1}}\\
  &\lesssim(\|(\nabla  Z_2,W)\|_{\dot\B^{\frac dp}_{p,1}}+\|Z_2\|_{\dot\B^{\frac dp}_{p,1}}\|\nabla Z_1\|_{\dot\B^{\frac dp}_{p,1}})\|Z\|_{\dot\B^{\frac{d}{p}+1}_{p,1}}.
  \end{aligned}$$
   Bounding the terms $\wt A^k_{2,m}(Z)\d_t\d_kZ_m$  involves a commutator estimate.
  We write $$\ddj(\wt A^k_{2,m}(Z)\d_t\d_kZ_m)=\wt A^k_{2,m}(Z)\d_t\d_kZ_{m,j}-R^{m,1}_j\with R^{m,1}_j=[\wt A^k_{2,m}(Z),\ddj]\d_t\d_kZ_m.$$ 
  Now, combining H\"older's inequality, embedding and \eqref{eq:comparaison}   yields
  $$\begin{aligned} \sum_{j\leq J_1}2^{j\frac{d}{p}}\|\wt A^k_{2,m}(Z)\d_t\d_kZ_{m,j}\|_{L^p}
  &\lesssim \|Z\|_{L^\infty}\sum_{j\leq J_1}
  2^{j\frac dp}\|\d_t\d_kZ_{m,j}\|_{L^p}\\
  &\lesssim 
 \|Z\|_{\dot\B^{\frac dp}_{p,1}}\|\d_t Z\|_{\dot\B^{\frac dp}_{p,1}} \\
  &\lesssim (\|(\nabla  Z_2,W)\|_{\dot\B^{\frac dp}_{p,1}}+\|Z_2\|_{\dot\B^{\frac dp}_{p,1}}\|\nabla Z_1\|_{\dot\B^{\frac dp}_{p,1}})\|Z\|_{\dot\B^{\frac{d}{p}}_{p,1}}
  \end{aligned}$$
and using \eqref{eq:com1}, we obtain
   $$\begin{aligned} 
   \sum_{j\leq J_1}   2^{j\frac dp}\|R_j^{m,1}\|_{L^p} &\lesssim 
 \|\d_t Z\|_{\dot\B^{\frac dp}_{p,1}} \|\nabla Z\|_{\dot\B^{\frac dp}_{p,1}}\\
  &\lesssim (\|(\nabla  Z_2,W)\|_{\dot\B^{\frac dp}_{p,1}}+\|Z_2\|_{\dot\B^{\frac dp}_{p,1}}\|\nabla Z_1\|_{\dot\B^{\frac dp}_{p,1}})\|Z\|_{\dot\B^{\frac{d}{p}+1}_{p,1}}.
  \end{aligned}$$
Differentiating \eqref{eq:WW} and \eqref{eq:Z1}, 
then using product laws (namely Proposition \ref{LP}) and \eqref{eq:bernstein} yields
$$\begin{aligned}
\|(\d_t\nabla Z_1,\d_t\nabla Z_2)\|^\ell_{\dot\B^{\frac dp}_{p,1}}&\lesssim 
\|W\|^\ell_{\dot\B^{\frac dp+1}_{p,1}}+\|\nabla Z_2\|_{\dot\B^{\frac dp+1}_{p,1}}^\ell
+\|Z_2\|_{\dot\B^{\frac dp}_{p,1}}\|\nabla Z_1\|_{\dot\B^{\frac dp}_{p,1}}+\|Z\|_{\dot\B^{\frac dp}_{p,1}}\|\nabla Z_2\|_{\dot\B^{\frac dp}_{p,1}}.\end{aligned}
$$
Hence, using Lemma \ref{SimpliCarre}, multiplying the resulting inequality  by $2^{j\frac{d}{p}}$, summing up on $j\leq J_1+1$ and taking advantage of  \eqref{eq:smallZ}, we end up with \eqref{est:W}.
 \end{proof}

\subsubsection*{Step 2: Estimates for  \texorpdfstring{$Z_1$}{TEXT}} We have the following proposition.

\begin{Prop}\label{PropZ1} There exists come  $\kappa_0>0$ such that
\begin{multline}\label{est:Z1}
\|Z_1(t)\|^\ell_{\dot\B^{\frac dp}_{p,1}}
  + \kappa_0\int_0^t  \|Z_1\|^\ell_{\dot\B^{\frac dp+2}_{p,1}}
 \leq  \|Z_{1,0}\|^\ell_{\dot\B^{\frac dp}_{p,1}}\\
 +C\biggl(\int_0^t\|(W,\nabla Z_2)\|^\ell_{\dot\B^{\frac dp+1}_{p,1}}
  +\int_0^t\|Z_2\|_{\dot\B^{\frac dp}_{p,1}}\|Z\|_{\dot\B^{\frac dp+1}_{p,1}}+\int_0^t\|Z\|_{\dot\B^{\frac dp+1}_{p,1}}\|W\|_{\dot\B^{\frac dp}_{p,1}}+\int_0^t\|Z\|_{\dot\B^{\frac dp+1}_{p,1}}^2\biggr)\cdotp
\end{multline}
\end{Prop}

\begin{proof}
Applying $\ddj$ to \eqref{Z1withW}, taking the scalar product with $|Z_{1,j}|^{p-2}Z_{1,j}$ 
and using \eqref{Elliptic} gives for some $\kappa_0>0,$
$$\frac{1}{p}\frac{d}{dt}\|Z_{1,j}\|_{L^p}^p+\kappa_02^{2j}\| Z_{1,j}\|_{L^p}^p \lesssim
\|(f_{1,j},f_{2,j},f_{3,j},f_{4,j},f_{5,j})\|_{L^p}\|Z_{1,j}\|_{L^p}^{p-1}$$
 where the $f_i$'s are defined in \eqref{Z1withW}. 
\smallbreak
Using Lemma \ref{SimpliCarre}, multiplying by $2^{j\frac dp}$
then summing on $j\leq J_1,$ this easily leads to
\begin{equation}\label{eq:fii}\|Z_1(t)\|^\ell_{\dot\B^{\frac dp}_{p,1}}
  + \kappa_0\int_0^t  \|Z_1\|^\ell_{\dot\B^{\frac dp+2}_{p,1}}
 \lesssim  \|Z_0\|^\ell_{\dot\B^{\frac dp}_{p,1}}+\|(f_1,f_2,f_3,f_4,f_5)\|^\ell_{\dot\B^{\frac dp}_{p,1}}.
 \end{equation} 
 In order to bound the term corresponding to $f_1,$ 
 we use the  decomposition:
\begin{multline}\label{eq:f1}
f_1=\sum_{k=1}^d\sum_{m=1}^d\Bigl(
\bar A^k_{1,2}L_2^{-1}\bar A^k_{2,2}\d_k\d_m Z_2
+\bar A^k_{1,2}\d_k\bigl(L_2^{-1}
\wt A^m_{2,2}(Z)\d_m Z_2\bigr)\\+\wt A^k_{1,2}(Z)L_2^{-1}\bar A^m_{2,2}\d_k\d_m Z_2
+\wt A^k_{1,2}(Z)\d_k\bigl(L_2^{-1}
\wt A^m_{2,2}(Z)\d_m Z_2\bigr)\Bigr)\cdotp
\end{multline}
The first term in the right-hand side term 
obviously satisfies
$$\begin{aligned}
\|\bar A^k_{1,2}L_2^{-1}\bar A^k_{2,2}\d_k\d_m Z_2\|^\ell_{\dot\B^{\frac dp}_{p,1}} &\lesssim \|Z_2\|^\ell_{\dot\B^{\frac dp+2}_{p,1}}.\end{aligned}$$
For the other terms of \eqref{eq:f1},  applying directly the product laws of Proposition \ref{LP} would entail 
a loss of one derivative. To avoid it, we will take advantage
once more of the commutator estimate provided by Lemma
\ref{CP}. Let us explain how to proceed for the 
last term of \eqref{eq:f1} (which is the most complicated).
We have 
$$\begin{aligned}\wt A^k_{1,2}(Z)
\d_k(L_2^{-1}\wt A^m_{2,2}(Z)\d_m Z_2)=\wt A^k_{1,2}(Z)L_2^{-1}
\d_k(\wt A^m_{2,2}(Z))\d_m Z_2+
\wt A^k_{1,2}(Z)L_2^{-1}
\wt A^\ell_{2,2}(Z)\d_k\d_m Z_2. \end{aligned}$$
The first term can be bounded directly 
as follows: 
$$
\|\wt A^k_{1,2}(Z)L_2^{-1}
\d_k(\wt A^m_{2,2}(Z))\d_m Z_2\|^\ell_{\dot\B^{\frac dp}_{p,1}} \lesssim \|Z\|_{\dot\B^{\frac dp}_{p,1}}
\|\nabla Z\|_{\dot\B^{\frac dp}_{p,1}}^2.$$
To handle  the second term, one introduces a  commutator
with $\ddj$ as follows: 
$$
\ddj\Bigl(\wt A^k_{1,2}(Z)L_2^{-1}
\wt A^\ell_{2,2}(Z)\d_k\d_mZ_2\Bigr)=
\wt A^k_{1,2}(Z)L_2^{-1}\wt A^m_{2,2}(Z)\d_k\d_m Z_{2,j}-R^{m,2}_j
 $$
 where $R^{m,2}_j\triangleq
 \left[\wt A^k_{1,2}(Z)L_2^{-1}\wt A^m_{2,2}(Z),\ddj\right]\d_k\d_m Z_2.$
\medbreak
 We have
$$\begin{aligned}\sum_{j\leq J_1}2^{j\frac{d}{p}}
\|\wt A^k_{1,2}(Z)L_2^{-1}\wt A^\ell_{2,2}(Z)\d_k\d_\ell Z_{2,j}\|_{L^p}&\lesssim
\sum_{j\leq J_1}2^{j\frac{d}{p}} \|Z\|_{L^\infty}^2
\|\nabla^2Z_{2,j}\|_{L^p}\\&\lesssim  \|Z\|_{\dot\B^{\frac dp}_{p,1}}^2\|Z_2\|^\ell_{\dot\B^{\frac dp+2}_{p,1}},
  \end{aligned}$$
and \eqref{eq:com1} yields
 $$\begin{aligned} 
 \sum_{j\leq J_1}2^{j\frac dp} \| R_j^{\ell,2}\|_{L^p} &\lesssim \|\nabla (Z\otimes Z)\|_{\dot\B^{\frac dp}_{p,1}}
 \|\nabla Z_2\|_{\dot\B^{\frac dp}_{p,1}} \\
  &\lesssim \|Z\|_{\dot\B^{\frac dp}_{p,1}}\|Z\|_{\dot\B^{\frac dp+1}_{p,1}}
 \|Z_2\|_{\dot\B^{\frac dp+1}_{p,1}}. 
  \end{aligned}$$  
  
  The other terms of \eqref{eq:f1} may be handled similarly. 
  In the end, using \eqref{eq:smallZ},  we obtain
$$\|f_1\|^\ell_{\dot\B^{\frac dp}_{p,1}} \lesssim \|Z_2\|^\ell_{\dot\B^{\frac dp+2}_{p,1}}+\| Z\|_{\dot\B^{\frac dp+1}_{p,1}}^2.$$
Noticing that the component $Z_2$ did not play a special role in the above computations, we can reproduce 
the procedure for $f_3$  and $f_4$, and eventually get:
$$\begin{aligned}
\|f_3\|^\ell_{\dot\B^{\frac dp}_{p,1}}&=\|\sum_{k=1}^d\sum_{m=1}^d A_{1,2}^k(V)\partial_k(L_2^{-1}\wt A^{m}_{2,1}(Z)\d_mZ_1)\|^\ell_{\dot\B^{\frac dp}_{p,1}}\\&\lesssim \| Z\|_{\dot\B^{\frac dp+1}_{p,1}}^2+\|Z\|_{\dot\B^{\frac dp}_{p,1}}\|Z_1\|^\ell_{\dot\B^{\frac dp+2}_{p,1}},\\[1ex]
\|f_4\|_{\dot\B^{\frac dp}_{p,1}}^\ell &\lesssim\|\sum_{k=1}^d\sum_{m=1}^d \wt A_{1,2}^k(Z)L_2^{-1}A^{m}_{2,1}\d_k\d_m Z_1\|^\ell_{\dot\B^{\frac dp}_{p,1}}\\ &\lesssim 
\| Z\|_{\dot\B^{\frac dp+1}_{p,1}}^2+\|Z\|_{\dot\B^{\frac dp}_{p,1}}\|Z_1\|^\ell_{\dot\B^{\frac dp+2}_{p,1}}\cdotp
\end{aligned}$$
For $f_2$, we write that 
$$f_2=-\sum_{k=1}^d
\bigl(\bar A_{1,2}^k\d_kW-\wt A_{1,2}^k(Z)\d_kW\bigr)\cdotp$$
Hence, 
$$\ddj f_2=-\sum_{k=1}^d \bar A_{1,2}^k\d_kW_j
-\wt A_{1,2}^k(Z)\d_kW_j
+[\wt A_{1,2}^k(Z),\ddj]\d_kW.$$
This allows to get
$$\begin{aligned}
\|f_2\|^\ell_{\dot\B^{\frac dp}_{p,1}}&\lesssim\|\sum_{k=1}^d\bar A_{1,2}^k\partial_{k}W^\ell \|_{\dot\B^{\frac dp}_{p,1}}+\|Z\|_{\dot\B^{\frac dp}_{p,1}}\| W\|^\ell_{\dot\B^{\frac dp+1}_{p,1}}+\|\nabla Z\|_{\dot\B^{\frac dp}_{p,1}}\| W\|_{\dot\B^{\frac dp}_{p,1}}\\
&\lesssim\|W\|^\ell_{\dot\B^{\frac dp+1}_{p,1}}
+ \|Z\|_{\dot\B^{\frac dp+1}_{p,1}}\|W\|_{\dot\B^{\frac dp}_{p,1}}.\end{aligned}$$
 The structure condition \eqref{StructAssum}  comes into play only for bounding $f_5.$
Thanks to it and to  \eqref{eq:smallZ}, we get
$$ \|f_5\|^\ell_{\dot\B^{\frac dp}_{p,1}} \lesssim \|Z_2\|_{\dot\B^{\frac dp}_{p,1}}\|\nabla Z_1\|_{\dot\B^{\frac dp}_{p,1}}.$$
Inserting the estimates pertaining to $f_1,f_2,f_3,f_4,f_5$ in \eqref{eq:fii} and using \eqref{eq:smallZ} completes the proof of the proposition.
\end{proof} Now, choosing $J_1$ small enough (so that the higher-order linear terms of the right-hand side are absorbed by the left-hand side) and putting together \eqref{est:W} and \eqref{est:Z1}, we obtain
\begin{multline}\label{est:WZ_1}
\|(W,Z_1)\|^\ell_{L^\infty_t(\dot\B^{\frac dp}_{p,1})}
  + \kappa_0\int_0^t\biggl(\|Z_1\|^\ell_{\dot\B^{\frac dp+2}_{p,1}}+
\|W\|_{\dot{\mathbb{B}}^{\frac{d}{p}}_{p,1}}^\ell\biggr)
 \leq  \|(W_0,Z_{1,0})\|^\ell_{\dot\B^{\frac dp}_{p,1}}\\
 +C\biggl(\int_0^t\|Z_2\|^\ell_{\dot\B^{\frac dp+2}_{p,1}}+\int_0^T\|Z_2\|_{\dot\B^{\frac dp}_{p,1}}\|Z\|_{\dot\B^{\frac dp+1}_{p,1}}+\int_0^t\|(\nabla Z_2,W)\|_{\dot\B^{\frac dp}_{p,1}} \|Z\|_{\dot\B^{\frac dp}_{p,1}\cap \dot\B^{\frac dp+1}_{p,1}}+\int_0^T\|Z\|_{\dot\B^{\frac dp+1}_{p,1}}^2\biggr)
 \cdotp
\end{multline}

\subsubsection*{Step 3: Recovering information for $Z_2$.}
In order to bound $Z_2$ from $W,$ one can use the identity:
\begin{equation}\label{eq:WZ2}W-Z_2=L_2^{-1}\sum_{k=1}^d\Bigl(\bar A^k_{2,1}\d_kZ_1+\wt A^k_{2,1}(Z)\d_kZ_1
+\bar A^k_{2,2}\d_kZ_2+\wt A^k_{2,2}(Z)\d_kZ_2\Bigr)\cdotp\end{equation}
It implies that
$$\begin{aligned}
\|W- Z_2\|_{\dot\B^{\frac dp}_{p,1}}^\ell
&\lesssim \|\nabla  Z\|_{\dot\B^{\frac dp}_{p,1}}^\ell 
+ \|Z\|_{\dot\B^{\frac dp}_{p,1}}\|\nabla Z\|_{\dot\B^{\frac dp}_{p,1}}\\
&\lesssim(1+ \|Z\|_{\dot\B^{\frac dp}_{p,1}})\|\nabla  Z\|_{\dot\B^{\frac dp}_{p,1}}^\ell 
+ \|Z\|_{\dot\B^{\frac dp}_{p,1}}\|Z\|_{\dot\B^{\frac d2+1}_{2,1}}^h.
\end{aligned}$$
Hence, owing to \eqref{eq:smallZ}, if one takes $J_1$ negative enough,
\begin{equation}\label{eq:WZ2l8b}
\|Z_2\|_{\dot\B^{\frac dp}_{p,1}}^\ell \lesssim
\|W\|_{\dot\B^{\frac dp}_{p,1}}^\ell +
\|\nabla Z_1\|_{\dot\B^{\frac dp}_{p,1}}^\ell
+\|Z\|_{\dot\B^{\frac dp}_{p,1}}
\|Z\|_{\dot\B^{\frac d2+1}_{2,1}}^h.
\end{equation}
This already ensures  that
$\|Z_2(t)\|_{\dot\B^{\frac dp}_{p,1}}$ may be bounded 
for all $t\in[0,T]$ by the right-hand side of \eqref{eq:APLP}. 
Next, since for $m=1,2,$ owing to \eqref{eq:bernstein} 
and the linearity
of $A^k,$ it holds that
\begin{align*}
\|\wt A^k_{2,1}(Z)\d_kZ_m\|_{\dot\B^{\frac dp+1}_{p,1}}^\ell
\lesssim& \|\wt A^k_{2,m}(Z)\d_kZ_m^\ell\|_{\dot\B^{\frac dp+1}_{p,1}}^\ell 
+ \|\wt A^k_{2,m}(Z)\d_kZ_m^h\|_{\dot\B^{\frac dp}_{p,1}}^\ell
\\ \lesssim&\|Z\|_{\dot\B^{\frac dp}_{p,1}}\|\nabla Z_m^\ell\|_{\dot\B^{\frac dp+1}_{p,1}}+\|Z\|_{\dot\B^{\frac dp+1}_{p,1}}\|\nabla Z_m^\ell\|_{\dot\B^{\frac dp}_{p,1}}+\|Z\|_{\dot\B^{\frac dp}_{p,1}}\|\nabla Z_m^h\|_{\dot\B^{\frac dp}_{p,1}}
\end{align*}
we have
\begin{equation}\label{eq:WZ2l1}
\|W- Z_2\|_{\dot\B^{\frac dp+1}_{p,1}}^\ell
\lesssim \|\nabla  Z\|_{\dot\B^{\frac dp+1}_{p,1}}^\ell (1+\|Z\|_{\dot\B^{\frac dp}_{p,1}})
+\|Z\|_{\dot\B^{\frac dp+1}_{p,1}}^2+\|Z\|_{\dot\B^{\frac dp}_{p,1}}\|Z\|^h_{\dot\B^{\frac d2+1}_{2,1}}.\end{equation}
Hence, one can also  include 
$\|Z_2\|^\ell_{L_t^1(\dot\B^{\frac dp+1}_{p,1})}$ in the left-hand side of \eqref{eq:APLP}, which completes the proof of Proposition 
\ref{APLP}.


\subsubsection{High frequencies analysis}

Although the functional framework for high frequencies  is the same as in \cite{CBD2}, one cannot 
repeat exactly the computations therein since the non-linear terms contain
 a little amount of low frequencies of $Z$ that are only in spaces  of the type $\dot{\mathbb{B}}^{s}_{p,1}$ for some $p>2$ and thus not in some $\dot{\mathbb{B}}^{s'}_{2,1}$.
 To overcome the difficulty, we 
 will have to resort to more elaborate product laws
 and commutator estimates  (see the Appendix). 
 
 The goal of this part is to prove the following proposition. 
  \begin{Prop} \label{EnergyHf}
  Let $p\in[2,4]$ if $d\leq4$ ($p\in[2,2d/(d-2)]$ if $d\geq5$) and define $p^*$ by the relation $1/p+1/p^*=1/2.$
  Then, the following a priori estimate holds:
$$\displaylines{
\norme{Z(t)}^h_{\dot{\mathbb{B}}^{\frac{d}{2}+1}_{2,1}}+\int_0^t\norme{Z}^h_{\dot{\mathbb{B}}^{\frac{d}{2}+1}_{2,1}}\lesssim\norme{Z_0}^h_{\dot{\mathbb{B}}^{\frac{d}{2}+1}_{2,1}} +\int_0^t\Bigl((\|Z\|_{\dot\B^{\frac d2+1}_{2,1}}^h\!+\!  \|Z\|_{\dot\B^{\frac dp}_{p,1}}^\ell)\|Z\|_{\dot\B^{\frac d2+1}_{2,1}}^h
+\norme{Z}_{\dot{\mathbb{B}}^{\frac{d}{p}}_{p,1}}\norme{ Z}^\ell_{\dot{\mathbb{B}}^{\frac{d}{p}+2}_{p,1}}\Bigr)\cdotp}$$
\end{Prop}
  \begin{proof}
  The starting point is to  differentiate in time the following functional:
\begin{equation}\label{eq:LqHF}  \cL_j 
\triangleq \norme{Z_j}^2_{L^2}+2^{-j} \cI_j,\qquad j\geq J_1,\end{equation} 
with
\begin{equation}\label{def:Iq}
\mathcal{I}_j\triangleq\int_{\mathbb{R}^d}\sum_{q=1}^{n-1}\varepsilon_q
\Im\left((L M_\omega^{q-1}\widehat{Z_j})\cdotp( L M_\omega^q\widehat{Z_j})\right),\end{equation}
and where  $\varepsilon_1,\dotsm,\varepsilon_{n-1}>0$ are chosen small enough. 
\medbreak
We note that
\begin{equation}\label{eq:cLjdt}
\frac d{dt}\cL_j=\frac d{dt}\|Z_j\|_{L^2}^2+2^{-j}
\frac d{dt}\cI_j.\end{equation}
 
 \subsubsection*{Step 1: Energy estimates.}
To  bound the first term in the right-hand side of \eqref{eq:cLjdt}, we localize  System \eqref{GE}  by means of $\ddj$
as follows:
$$\d_tZ_{j}+\sum_{k=1}^d\dot{S}_{j-1}A^k(V)\d_kZ_{j}+ LZ_j=R^1_j\with 
 R_j^1\triangleq\sum_{k=1}^d\dot{S}_{j-1}A^k(V)\d_kZ_{j}-\ddj(A^k(V)\d_kZ).$$
For $j\geq J_1$, taking the scalar product in $L^2(\R^d;\R^n)$ with 
$Z_j,$ then integrating by parts and using \eqref{partdissip2}
leads to:
\begin{equation}
\frac12\frac{d}{dt}\norme{Z_j}^2_{L^2}+\kappa_0\norme{Z_{2,j}}^2_{L^2}\lesssim
\norme{\nabla Z}_{L^\infty}\norme{Z_j}^2_{L^2}+\norme{R^1_j}_{L^2}\norme{Z_j}_{L^2}.\label{L2estimateHF}
\end{equation}

For the remainder term $R^1_j$, using Lemma \ref{CP} with 
$w=\wt A^k(Z),$ $z=\d_kZ,$
$k=0$, $\sigma_1=\frac{d}{p}+2$ and $\sigma_2=\frac{d}{p}+1$, 
we have (here we use our assumptions on $p$):
$$\displaylines{
\norme{R^1_j}_{L^2}\leq C c_j2^{-j(\frac{d}{2}+1)}
\sum_{k=1}^d
\Bigl(\norme{\nabla \wt A^k(Z)}_{\dot{\mathbb{B}}^{\frac dp}_{p,1}}\norme{Z}^h_{\dot{\mathbb{B}}^{\frac{d}{2}+1}_{2,1}}
+ \norme{\nabla Z}^\ell_{\dot{\mathbb{B}}^{\frac{d}{p}-\frac{d}{p*}}_{p,1}}\norme{\wt A^k(Z)}^\ell_{\dot{\mathbb{B}}^{\frac{d}{p}+2}_{p,1}}\hfill\cr\hfill+\norme{\nabla Z}_{\dot{\mathbb{B}}^{\frac dp}_{p,1}}\norme{\wt A^k(Z)}^h_{\dot{\mathbb{B}}^{\frac{d}{2}+1}_{2,1}}
+ \norme{\nabla Z}^\ell_{\dot{\mathbb{B}}^{\frac{d}{p}+1}_{p,1}}\norme{ \nabla\wt A^k(Z)}^\ell_{\dot{\mathbb{B}}^{\frac{d}{p}-\frac{d}{p*}}_{p,1}}\Bigr),}$$
whence, observing that $p^*\geq d$ and thus by \eqref{eq:comparaison},
\begin{equation}\label{eq:compbis}
\|\nabla Z\|^\ell_{\dot\B^{\frac dp-\frac d{p^*}}_{p,1}}
\lesssim \|Z\|^\ell_{\dot\B^{\frac dp}_{p,1}},\end{equation}
then using the linearity of $\wt A^k,$ we get
\begin{equation}\label{eq:R1}
\norme{R^1_j}_{L^2}\leq Cc_j2^{-j(\frac{d}{2}+1)}
\Bigl(\norme{Z}_{\dot{\mathbb{B}}^{\frac{d}{p}+1}_{p,1}}\norme{ Z}^h_{\dot{\mathbb{B}}^{\frac{d}{2}+1}_{2,1}}
+ \norme{Z}^\ell_{\dot{\mathbb{B}}^{\frac{d}{p}}_{p,1}}\norme{ Z}^\ell_{\dot{\mathbb{B}}^{\frac{d}{p}+2}_{p,1}}\Bigr)\cdotp
\end{equation}


\subsubsection*{Step 2: Cross estimates.}

To recover  the dissipation for all  the components of $Z$, we have to look 
at the time derivative of $\cI_j$ defined in \eqref{def:Iq}. 

To start with, let us rewrite \eqref{GE} as  follows: 
\begin{equation} 
\d_t Z+\sum_{k=1}^d\bar A^k\d_kZ+LZ= 
-\sum_{k=1}^d\wt A^k(Z)\d_kZ.
\end{equation} 
Hence, localizing the above equation yields
\begin{equation} \label{controlegradient}
\d_t Z_j+\sum_{k=1}^d\bar A^k\d_kZ_j+LZ_j= G_j\triangleq
-\sum_{k=1}^d\wt A^k(Z)\d_kZ_j+\sum_{k=1}^d[\wt A_k(Z),\ddj]\d_kZ.
\end{equation} 
Following the computations  we did in \cite{CBD2} leads for a suitable choice of 
 $\varepsilon_1,\cdots,\varepsilon_{n-1}$ to
\begin{equation}\label{eq:Iqbis}\frac{d}{dt}\mathcal{I}_j
+\frac{2^j}2 \sum_{q=1}^{n-1}\varepsilon_q \int_{\R^d} |LM_\omega^q\wh Z_j|^2\,d\xi
\leq \frac {2^{-j}\kappa_0}2 \|LZ_j\|_{L^2}^2 +  C\|\ddj G\|_{L^2}\|Z_j\|_{L^2}.\end{equation}


We have 
$$2^{j\frac d2}\|G_j\|_{L^2} \lesssim 
\|\wt A^k(Z)\|_{L^\infty}\bigl(2^{j\frac d2}\|\nabla Z_j\|_{L^2}\bigr)
+\sum_{k=1}^d 2^{j\frac d2}\|[\wt A_k(Z),\ddj]\d_kZ\|_{L^2}.
$$
Hence, using the embedding $\dot\B^{\frac d2}_{2,1}\hookrightarrow L^\infty,$
 applying Lemma \ref{CP} with $k=0$, $\sigma_1=d/p+2$ and $\sigma_2=\frac{d}{p}+1$ and 
 remembering that all the $\wt A^k$'s are linear, we get
$$\displaylines{\sum_{j\geq J_1}2^{j\frac d2}\|G_j\|_{L^2} \lesssim \bigl\|Z\|_{\dot{\mathbb{B}}^{\frac dp}_{p,1}}\|\nabla Z\|_{\dot{\mathbb{B}}^{\frac d2}_{2,1}}^h+
\norme{\nabla Z}_{\dot{\mathbb{B}}^{\frac dp}_{p,1}}\norme{Z}^h_{\dot{\mathbb{B}}^{\frac{d}{2}}_{2,1}}\hfill\cr\hfill
+ \norme{Z}^\ell_{\dot{\mathbb{B}}^{\frac{d}{p}-\frac{d}{p*}}_{p,1}}\norme{\nabla Z}^\ell_{\dot{\mathbb{B}}^{\frac{d}{p}+1}_{p,1}}
+ \norme{Z}^\ell_{\dot{\mathbb{B}}^{\frac{d}{p}+2}_{p,1}}\norme{ \nabla  Z}^\ell_{\dot{\mathbb{B}}^{\frac{d}{p}-\frac{d}{p*}}_{p,1}},}
$$
whence,   owing to \eqref{eq:compbis},
\begin{equation}\label{eq:G1} \norme{G}^h_{\dot{\mathbb{B}}^{\frac d2}_{2,1}}\lesssim\bigl\| Z\|_{\dot{\mathbb{B}}^{\frac d2+1}_{2,1}}^h(\|Z\|^\ell_{\dot{\mathbb{B}}^{\frac dp}_{p,1}}+\|Z\|^h_{\dot{\mathbb{B}}^{\frac dp+1}_{p,1}})+\norme{Z}^\ell_{\dot{\mathbb{B}}^{\frac{d}{p}}_{p,1}}\norme{Z}^\ell_{\dot{\mathbb{B}}^{\frac{d}{p}+2}_{p,1}}.
\end{equation}

Remember that since the (SK) condition is satisfied, the quantity $\cN_{\bar V}$
defined in \eqref{eq:Nomega}  is positive for any choice 
of positive parameters $\e_0,\cdots,\e_{n-1}.$ 
Consequently, if  we set 
$$\cH_j:= \frac{\kappa_0}2\|LZ_{j}\|^2 + \eta\sum_{q=1}^{n-1}\varepsilon_q \int_{\R^d} |LM_\omega^q \wh Z_j|^2\,d\xi$$
and use Fourier-Plancherel theorem and that $\mathcal{L}_j\simeq \|Z_j\|_{L^2}$, we see that 
(up to a change of  $\kappa_0$ and choosing $\eta$ small enough 
to kill the first term of the right-hand side of \eqref{eq:Iqbis}), we have  for all $j\geq J_1,$ 
\begin{equation}\label{eq:cHq}\cH_j\geq \kappa_0\cL_j.\end{equation}

Combining Inequalities  \eqref{L2estimateHF} and
\eqref{eq:R1}, the cross estimate \eqref{eq:Iqbis}, \eqref{eq:G1} and \eqref{eq:cHq}, we get
 $$\frac d{dt}2^{j(\frac d2+1)}\cL_j + \kappa_0 2^{j(\frac d2+1)}\cL_j
 \leq C c_j\Bigl(\norme{(\nabla Z,Z)}_{\dot{\mathbb{B}}^{\frac{d}{p}}_{p,1}}\norme{Z}^h_{\dot{\mathbb{B}}^{\frac{d}{2}+1}_{2,1}}\norme{Z_j}_{L^2}+\norme{Z}_{\dot{\mathbb{B}}^{\frac{d}{p}}_{p,1}}\norme{ Z}^\ell_{\dot{\mathbb{B}}^{\frac{d}{p}+2}_{p,1}}\norme{Z_j}_{L^2}\Bigr)\cdotp$$
Hence, using that $\mathcal{L}_j\simeq \|Z_j\|_{L^2}$ and Lemma \ref{SimpliCarre},   and summing up on 
 $j\geq J_1$ yields
 \begin{multline}\label{Zh}\|Z(t)\|^h_{\dot\cB^{\frac d2+1}_{2,1}}
  +\kappa_0\int_0^t  \|Z\|^h_{\dot\cB^{\frac d2+1}_{2,1}}\\\leq \|Z_0\|^h_{\dot\cB^{\frac d2+1}_{2,1}}+
C\int_0^t\Bigl((\|Z\|_{\dot\B^{\frac d2+1}_{2,1}}^h\!+\!  \|Z\|_{\dot\B^{\frac dp}_{p,1}}^\ell)\|Z\|_{\dot\B^{\frac d2+1}_{2,1}}^h
+\norme{Z}_{\dot{\mathbb{B}}^{\frac{d}{p}}_{p,1}}\norme{ Z}^\ell_{\dot{\mathbb{B}}^{\frac{d}{p}+2}_{p,1}}\Bigr)
\end{multline}
where we used the notation
$$\|Z\|^h_{\dot\cB^{\sigma}_{2,1}}\triangleq   \sum_{j\geq J_1} 2^{j\sigma} \sqrt{\cL_j}.$$
As $\|Z\|^h_{\dot\cB^{\sigma}_{2,1}}\simeq \|Z\|^h_{\dot\B^{\sigma}_{2,1}},$
this completes the proof of the proposition.
  \end{proof}


  \subsubsection{The final a priori estimate}
  
  As a first, observe that \eqref{eq:WZ2} implies that
  $$
  \|W-Z_2\|_{\dot\B^{\frac dp}_{p,1}}^h
\lesssim \|\nabla Z\|^h_{\dot\B^{\frac dp}_{p,1}}
+\sum_{m=1}^2\bigl(\|\wt A^k_{2,m}(Z)\d_kZ_m^\ell\|_{\dot\B^{\frac dp+1}_{p,1}}
+\|\wt A^k_{2,1}(Z)\d_kZ_m^h\|_{\dot\B^{\frac dp}_{p,1}}\bigr)\cdotp
$$  
  Hence, using product laws in Besov spaces (that is Proposition \ref{LP}), embedding, the smallness
  condition \eqref{eq:smallZ} 
  and \eqref{eq:bernstein}, we get
  \begin{align}
  \|W-Z_2\|_{\dot\B^{\frac dp}_{p,1}}^h &\lesssim \|\nabla Z\|^h_{\dot\B^{\frac dp}_{p,1}}
+\|Z\|_{\dot\B^{\frac dp}_{p,1}}\|Z\|_{\dot\B^{\frac dp+2}_{p,1}}^\ell
+\|Z\|_{\dot\B^{\frac dp+1}_{p,1}}^2
+\|Z\|_{\dot\B^{\frac dp}_{p,1}}\|\nabla Z\|_{\dot\B^{\frac dp}_{p,1}}^h\nonumber\\\label{eq:WZ2H}
&\lesssim \|\nabla Z\|^h_{\dot\B^{\frac d2}_{2,1}}
+\|Z\|_{\dot\B^{\frac dp}_{p,1}}\|Z\|_{\dot\B^{\frac dp+2}_{p,1}}^\ell
+\|Z\|_{\dot\B^{\frac dp+1}_{p,1}}^2.\end{align}
   Let us introduce the functionals   
\begin{equation}\label{def:cL}
\cL\triangleq \|(Z,W)\|^\ell_{\dot{\mathbb{B}}^{\frac dp}_{p,1}}
  + \|Z\|^h_{\dot\cB^{\frac d2+1}_{2,1}}
  \andf \cH\triangleq   \|Z\|^h_{\dot\B^{\frac d2+1}_{2,1}} +\|Z_1\|^\ell_{\dot\B^{\frac dp+2}_{p,1}}+\|Z_2\|^\ell_{\dot\B^{\frac dp+1}_{p,1}}+\|W\|_{\dot\B^{\frac dp}_{p,1}}.
    \end{equation}
   As seen in \eqref{eq:WZ2l8b}, $\cL$  is  equivalent to 
   $ \|Z\|^\ell_{\dot\B^{\frac dp}_{p,1}}
  + \|Z\|^h_{\dot\B^{\frac d2+1}_{2,1}}$.

  Adding up the inequality 
  of Proposition \ref{APLP} with Inequalities \eqref{Zh} and \eqref{eq:WZ2H},
   remembering \eqref{eq:smallZ} and using
   several times the fact that
   $$
   \|Z\|_{\dot\B^{\frac dp}_{p,1}}+
    \|Z\|_{\dot\B^{\frac dp+1}_{p,1}}  \lesssim
     \|Z\|^\ell_{\dot\B^{\frac dp}_{p,1}}
  + \|Z\|^h_{\dot\B^{\frac d2+1}_{2,1}},$$
    we get   for all $t\in[0,T],$ 
$$\displaylines{\cL(t) + \int_0^t \cH\lesssim \|Z_0\|^h_{\dot\B^{\frac d2+1}_{2,1}}+\|(W_0,Z_0)\|^\ell_{\dot\B^{\frac dp}_{p,1}}+\int_0^t(\|Z\|_{\dot\B^{\frac d2+1}_{2,1}}^h\!+\!  \|Z\|_{\dot\B^{\frac dp}_{p,1}}^\ell)\|Z\|_{\dot\B^{\frac d2+1}_{2,1}}^h
\hfill\cr\hfill +\int_0^t\norme{Z}_{\dot{\mathbb{B}}^{\frac{d}{p}+1}_{p,1}}^2
+\int_0^t\norme{Z}_{\dot{\mathbb{B}}^{\frac{d}{p}}_{p,1}}\norme{ Z_2}_{\dot{\mathbb{B}}^{\frac{d}{p}+1}_{p,1}}+\int_0^t\|Z_2\|_{\dot\B^{\frac dp}_{p,1}}\|Z\|_{\dot\B^{\frac dp+1}_{p,1}}
 +\int_0^t\|Z\|_{\dot\B^{\frac dp}_{p,1}\cap\dot\B^{\frac dp+1}_{p,1}}\|W\|_{\dot\B^{\frac dp}_{p,1}}.}$$
Owing to the definition of $\cL,$ an obvious embedding and \eqref{eq:smallZ},
the above inequality may be simplified into: 
\begin{multline}\label{eq:cLt}\cL(t) + \int_0^t \cH\lesssim \|Z_0\|^h_{\dot\B^{\frac d2+1}_{2,1}}+\|(W_0,Z_0)\|^\ell_{\dot\B^{\frac dp}_{p,1}}+\int_0^t\bigl(\|W\|_{\dot\B^{\frac dp}_{p,1}}+\|Z\|_{\dot\B^{\frac d2+1}_{2,1}}^h\bigr)\cL
\\+\int_0^t(\|Z\|^\ell_{\dot\B^{\frac dp+1}_{p,1}})^2
 +\int_0^t\|Z_2\|_{\dot\B^{\frac dp}_{p,1}}^\ell\|Z\|_{\dot\B^{\frac dp+1}_{p,1}}.\end{multline}
On the one hand,  by interpolation and \eqref{eq:comparaison}, we have
$$
(\|Z\|^\ell_{\dot\B^{\frac dp+1}_{p,1}})^2\leq
\|Z\|^\ell_{\dot\B^{\frac dp}_{p,1}}
\|Z\|^\ell_{\dot\B^{\frac dp+2}_{p,1}}\lesssim
\|Z\|_{\dot\B^{\frac dp}_{p,1}}
\bigl(\|Z_1\|^\ell_{\dot\B^{\frac dp+2}_{p,1}}+\|Z_2\|^\ell_{\dot\B^{\frac dp+1}_{p,1}}\bigr)
\lesssim
\|Z\|_{\dot\B^{\frac dp}_{p,1}}\cH,
$$
and the time integral of this term may thus be absorbed by the left-hand side of \eqref{eq:cLt}. 
\smallbreak
On the other hand, \eqref{eq:WZ2l8b} guarantees that
$$
\|Z_2\|^\ell_{\dot\B^{\frac dp}_{p,1}}\|Z\|_{\dot\B^{\frac dp+1}_{p,1}}
\lesssim \|W\|^\ell_{\dot\B^{\frac dp}_{p,1}}\|Z\|_{\dot\B^{\frac dp+1}_{p,1}}
+\bigl(\|Z\|_{\dot\B^{\frac dp+1}_{p,1}}^\ell\bigr)^2
+\bigl(\|Z\|_{\dot\B^{\frac dp+1}_{p,1}}^h\bigr)^2
+\|Z\|_{\dot\B^{\frac dp}_{p,1}}\|Z\|^h_{\dot\B^{\frac d2+1}_{2,1}}\|Z\|_{\dot\B^{\frac dp+1}_{p,1}}.
$$
The second term of the right-hand side may be handled as above, and we thus end up owing to \eqref{eq:smallZ}
with 
$$\|Z_2\|^\ell_{\dot\B^{\frac dp}_{p,1}}\|Z\|_{\dot\B^{\frac dp+1}_{p,1}}
\lesssim \bigl( \|W\|^\ell_{\dot\B^{\frac dp}_{p,1}}+\|Z\|^h_{\dot\B^{\frac d2+1}_{2,1}}\bigr)\cL
+\|Z\|_{\dot\B^{\frac dp}_{p,1}}\cH.$$
Hence, reverting to \eqref{eq:cLt}, we conclude that 
$$\cL(t) + \int_0^t \cH\leq C\biggl(\|Z_0\|^h_{\dot\B^{\frac d2+1}_{2,1}}+\|(W_0,Z_0)\|^\ell_{\dot\B^{\frac dp}_{p,1}}+
\int_0^t \cH\cL\biggr)\cdotp$$
It is now clear that if $\|Z_0\|^h_{\dot\B^{\frac d2+1}_{2,1}}+\|(W_0,Z_0)\|^\ell_{\dot\B^{\frac dp}_{p,1}}$ or, 
equivalently, $\|Z_0\|^h_{\dot\B^{\frac d2+1}_{2,1}}+\|Z_0\|^\ell_{\dot\B^{\frac dp}_{p,1}}$ is 
small enough, then we have 
\begin{equation}\label{eq:cL}\cL(t) + \int_0^t \cH\leq C\bigl(
\|Z_0\|^h_{\dot\B^{\frac d2+1}_{2,1}}+\|Z_0\|^\ell_{\dot\B^{\frac dp}_{p,1}}\bigr)\quad\hbox{for all }\ t\in[0,T].\end{equation}
In order to complete the proof of the estimate in Theorem \ref{ThmExistLp}, it suffices to observe that, in light 
of \eqref{eq:WZ2l8b}, one can recover a $L^2$-in-time control of $Z_2$
as follows:
$$\begin{aligned}\|Z_2\|_{L^2_T(\dot\B^{\frac{d}{p}}_{2,1})}^\ell
&\lesssim \|W\|_{L^2_T(\dot\B^{\frac{d}{p}}_{p,1})}^\ell+\|\nabla  Z\|_{L^2_T(\dot\B^{\frac dp}_{p,1})}^\ell 
+ \|Z\|_{L^\infty_T(\dot\B^{\frac dp}_{p,1})}\|\nabla Z\|^h_{L^2_T(\dot\B^{\frac d2}_{2,1})}\\
&\lesssim \bigl(\|W\|_{L^1_T(\dot\B^{\frac{d}{p}}_{p,1})}^\ell
\|W\|_{L^\infty_T(\dot\B^{\frac{d}{p}}_{p,1})}^\ell\bigr)^{1/2}
+\bigl(\|Z_1\|_{L^\infty_T(\dot\B^{\frac dp}_{p,1})}^\ell 
\|Z_1\|_{L^1_T(\dot\B^{\frac dp+2}_{p,1})}^\ell\bigr)^{1/2} 
\\&\hspace{6cm}+ \|Z\|_{L^\infty_T(\dot\B^{\frac dp}_{p,1})}
\bigl(\|\nabla Z\|^h_{L^\infty_T(\dot\B^{\frac d2}_{2,1})}
\|\nabla Z\|^h_{L^1_T(\dot\B^{\frac d2}_{2,1})}\bigr)^{1/2}.
\end{aligned}$$
Hence we have 
\begin{equation}\label{eq:Z2}\|Z_2\|_{L^2_T(\dot\B^{\frac{d}{p}}_{2,1})}
\leq C\bigl(
\|Z_0\|^h_{\dot\B^{\frac d2+1}_{2,1}}+\|Z_0\|^\ell_{\dot\B^{\frac dp}_{p,1}}\bigr)\cdotp
\end{equation}

\subsection{Proof of the existence part of Theorem \ref{ThmExistLp}}

Proving the existence of a global solution under 
the hypotheses of Theorem \ref{ThmExistLp} 
is an adaptation of \cite{CBD1} to the multi-dimensional case. 
First,  we multiply the  low frequencies of the data by a cut-off function  in order to have a converging sequence $(Z^n_0)_{n\in\N}$ of 
approximate data in the \emph{nonhomogeneous} Besov space
$\B^{\frac{d}{2}+1}_{2,1}.$ 
This enables us to take advantage of a classical existence statement 
(recalled in Appendix) to  construct
a sequence $(Z^n)_{n\in\N}$ of solutions to \eqref{GE}
with $\varepsilon=1$ and  initial data $(Z^n_0)_{n\in\N}.$
Then, from  the a priori estimates of the previous subsection,
embedding and a continuation criterion, we gather  that the approximate solutions are actually  global and  that $(Z_n)_{n\in\N}$ is bounded  in the space $E_p$
(that is $E_p^{J_\varepsilon}$ with $\varepsilon=1$).
At this stage, one may use compactness arguments in the spirit 
of Aubin-Lions lemma so as to prove convergence, up to subsequence
to a global solution of \eqref{GE}  supplemented with initial data $Z_0,$ 
with the desired properties.

\subsubsection*{First step: Construction of approximate solutions}

Let $Z_0$ be such that $Z_0^\ell\in\dot{\mathbb{B}}^{\frac{d}{p}}_{p,1}$ and $Z_0^h\in\dot{\mathbb{B}}^{\frac{d}{2}+1}_{2,1}$. 
Since $Z_0$ need not be in $\mathbb{B}^{\frac{d}{2}+1}_{2,1},$  we set for all $n\geq1,$
$$Z_{0}^n\triangleq   \chi_n\, \dot S_{J_1-5}Z_0+({\rm Id}-\dot S_{J_1-5})Z_0\with \chi_n\triangleq \chi(n^{-1}\cdot),$$
where $\chi$ stands  (for instance) for a  smooth function  with range in $[0,1],$ supported in  $]-4/3,4/3[$ and
such that $\chi\equiv1$ on $[-3/4,3/4].$ 
\smallbreak
 It is obvious that $Z_{0}^n$
 tends  to $Z_0$  in the sense of distributions, when $n$ tends to infinity.
Moreover, as $Z_0^\ell$ is in $\dot\B^{\frac dp}_{p,1},$  the  low frequencies of the data
are in $L^\infty,$ and the spatial truncation thus guarantees that  $Z_{0}^n\in {\mathbb{B}}^{\frac{d}{2}+1}_{2,1}.$ 
Furthermore, an easy adaptation of the proof of the one-dimensional case
in \cite{CBD1} reveals that  
\begin{equation}\label{eq:Zon}
\norme{Z_{0}^n}^\ell_{\dot{\mathbb{B}}^{\frac{d}{p}}_{p,1}}+ \norme{Z_{0}^n}^h_{\dot{\mathbb{B}}^{\frac{d}{2}+1}_{2,1}}\lesssim\norme{Z_0}^\ell_{\dot{\mathbb{B}}^{\frac{d}{p}}_{p,1}}+\norme{Z_0}^h_{\dot{\mathbb{B}}^{\frac{d}{2}+1}_{2,1}}.
\end{equation}

Now, applying Theorem \ref{ThmExistLocalLp}, we get a unique  maximal solution $Z^n$ in $\mathcal{C}([0,T_n[;\mathbb{B}^{\frac{d}{2}+1}_{2,1})\cap\mathcal{C}^1([0,T_n[;\mathbb{B}^{\frac{d}{2}}_{2,1})$ 
to System \eqref{GEQSYM}.

\subsubsection*{Second step: Uniform estimates} 

Since, for all $T>0,$ the space   $\mathcal{C}([0,T];\mathbb{B}^{\frac{d}{2}+1}_{2,1})\cap\mathcal{C}^1([0,T];\mathbb{B}^{\frac{d}{2}}_{2,1})$ is included in our `solution space' $E_p(T)$ (that is, $E_p$ restricted to $[0,T]$), 
one can take advantage of the computations for the previous sequence to bound our sequence.  In the end, denoting 
$$\displaylines{
X^n_p(t)\triangleq\norme{Z^n}^h_{L^\infty_t(\dot{\mathbb{B}}^{\frac{d}{2}+1}_{2,1})}
+\norme{Z^n}^\ell_{L^\infty_t(\dot{\mathbb{B}}^{\frac{d}{p}}_{p,1})}
+\norme{Z^n}^h_{L^1_t(\dot{\mathbb{B}}^{\frac{d}{2}+1}_{2,1})}
+\norme{Z_1^n}^\ell_{L^1_t(\dot{\mathbb{B}}^{\frac{d}{p}+2}_{p,1})}
\hfill\cr\hfill+\norme{Z_2^n}^\ell_{L^1_t(\dot{\mathbb{B}}^{\frac{d}{p}+1}_{p,1})}
+\norme{W^n}^\ell_{L^1_t(\dot{\mathbb{B}}^{\frac{d}{p}}_{p,1})}
+\norme{Z_2^n}^\ell_{L^2_t(\dot{\mathbb{B}}^{\frac{d}{p}}_{p,1})},}
$$
we get thanks to \eqref{eq:cL}, \eqref{eq:Z2} and  \eqref{eq:Zon}, 
\begin{equation}\label{Xp2}
 X^n_p(t)\leq  CX_{p,0}\quad\hbox{for all } \ t\in[0,T_n[.
\end{equation}
In order to show that the above inequality  implies that the solution is global (namely that $T_n=\infty$), one can argue by contradiction, assuming that $T_n<\infty,$ and use the blow-up criterion 
of Theorem \ref{ThmExistLocalLp}. However, we first have to justify that the nonhomogeneous Besov norm $\B^{\frac d2+1}_{2,1}$ of the solution is under control \emph{up to time $T_n.$}
Using the classical energy method for \eqref{GE} and the Gronwall lemma, we get that   for all $t<T_n,$ 
 $$\norme{Z^n(t)}_{L^2}\leq   C\norme{Z^n_0}_{L^2} \exp\biggl(C\int_0^t\norme{\nabla Z^n}_{L^\infty}\biggr)\cdotp$$ 
 Since \eqref{Xp2} and the embedding of $\dot\B^{\frac dp}_{p,1}$ and $\dot\B^{\frac d2}_{2,1}$ in $L^\infty$ ensure that $\nabla Z^n$ is in $L^1_{T_n}(L^\infty),$ we do have $Z^n\in L^\infty_{T_n}(L^2).$
 Combining with \eqref{Xp2} yields  $L^\infty_{T_n}({\mathbb{B}}^{\frac{d}{2}+1}_{2,1}).$ 

It is now easy to conclude : for all $t_{0,n}\in[0,T_n[,$  Theorem \ref{ThmExistLocalLp} provides us with  a solution of $(TM)$ with the initial data $Z(t_{0,n}$ 
 on $[t_{0,n},T+t_{0,n}]$ for some $T$ that may be bounded from below independently of   $t_{0,n}$. 
Consequently, choosing  $t_{0,n}$ such that $t_{0,n}>T_n-T$, we see that  the solution $Z^n$ can be extended beyond $T_n$, which 
contradicts the maximality of $T_n.$
Hence $T_n=\infty$ and the solution corresponding to the initial data $Z^n_0$ is global in time and satisfies $\eqref{Xp2}$ for all time.

\subsubsection*{Third step: Convergence} 

In order to show that $(Z^n)_{n\in\N}$ tends, up to subsequence, to some $Z\in E_p$, in the sense of distributions, that satisfies $\eqref{GE},$
one can use Ascoli Theorem and suitable compact embeddings in a similar fashion as in \cite{CBD1}. We omit the details here.

\subsection{Uniqueness}
 Since  our functional framework 
 is not the standard one for the low frequencies of the solution,  one cannot follow the classical  energy method like in e.g.  \cite{CBD2}. 
Here, for all $T>0,$  we shall estimate $\wt Z:=Z^1-Z^2$ in the space 
\begin{equation}\label{eq:FpT}
F_p(T)\triangleq \Bigl\{Z^\ell\in\cC([0,T];\dot\B^{\frac dp-\frac d{p*}}_{p,1})\,:\, 
Z^h\in \cC([0,T];\dot\B^{\frac d2}_{2,1})\Bigr\}
\cdotp\end{equation}
The reason for the exponent $d/2$ for high frequencies is the usual loss of one derivative when proving stability estimates for quasilinear hyperbolic systems. The exponent for low frequencies looks to be the best one 
for controlling  the nonlinearities.
To prove the uniqueness, we will need to following lemma.
\begin{Lemme} \label{lemma:uniq}
Let  $Z^1$ and $Z^2$ be two solutions of \eqref{GE} on $[0,T]$ supplemented by initial data $Z_{0}^1$ and $Z_{0}^2,$ respectively. Then,  $\wt Z\triangleq Z^1-Z^2$ satisfies the following 
a priori estimate for all $0\leq t\leq T$:
$$\displaylines{ \norme{\wt Z}^\ell_{L^\infty_t(\dot{\mathbb{B}}^{\frac{d}{p}-\frac{d}{p*}}_{p,1})}+\norme{\wt Z}_{L^\infty_t(\dot{\mathbb{B}}^{\frac{d}{2}}_{2,1})}^h\lesssim \norme{\wt Z_{0}}^\ell_{\dot{\mathbb{B}}^{\frac{d}{p}-\frac{d}{p*}}_{p,1}}+\norme{\wt Z_0}^h_{\dot{\mathbb{B}}^{\frac{d}{2}}_{2,1}}\hfill\cr\hfill+\int_0^t \Bigl(\|(Z^1,Z^2)\|^\ell_{\dot\B^{\frac dp}_{p,1}}
+\|(\nabla Z^1,\nabla Z^2)\|^h_{\dot{\mathbb{B}}^{\frac{d}{2}}_{2,1}}\Bigr)\bigl(\|{\wt Z}\bigr\|_{\dot{\mathbb{B}}^{\frac{d}{p}-\frac{d}{p*}}_{p,1}}^\ell+\|{\wt Z}\bigr\|_{\dot{\mathbb{B}}^{\frac{d}{2}}_{2,1}}^h\bigr)\cdotp}$$
\end{Lemme}
\begin{proof}  We have to proceed differently 
for estimating the low and the high frequencies of~$\wt Z.$

\smallbreak\noindent{\sl Step 1: Estimates for the low frequencies.} 
Let $V^1\triangleq \bar V+Z^1$ and $V^2\triangleq \bar V+Z^2.$ Observe that 
$\wt {Z}$ is a solution of
$$\displaylines{\quad
\d_t{\wt Z}+L\wt Z+\sum_{k=1}^dA^k(V^1)\d_k{\wt Z}=\sum_{k=1}^d\bigl(A^k(V^2)-A^k(V^1)\bigr)\d_kZ^2.}$$ 
Applying $\dot{\Delta}_j$, taking the scalar product with $|\wt Z_j|^{p-2}\wt Z_j$, integrating on $\mathbb{R}_+\times\mathbb{R}^d$ and using 
\eqref{partdissip2} and Lemma \ref{SimpliCarre}, we get for all $j\in\Z,$
$$\displaylines{\norme{\wt {Z}_j(t)}_{L^p}+\kappa_0\int_0^t \norme{L{\wt Z}_j}_{L^p}\leq\norme{\wt {Z}_{0,j}}_{L^p}+\sum_{k=1}^d\int_0^t \norme{\nabla A^k(V^1)}_{L^\infty}\norme{{\wt Z}_j}_{L^p} \hfill\cr\hfill
+ \int_0^t \biggl\|\dot{\Delta}_j\sum_{k=1}^d\bigl(A^k(V^2)-A^k(V^1)\bigr)\d_kZ^{2}\biggr\|_{L^p}+ \int_0^t
\sum_{k=1}^d\bigl\|[\ddj,A^k(V^1)]\d_k\wt Z\|_{L^p}.}
$$
Multiplying this inequality  
by $2^{j(\frac{d}{p}-\frac{d}{p*})}$ and using 
the embedding $\dot\B^{\frac dp}_{p,1}\hookrightarrow L^\infty$
as well as the commutator estimate \eqref{eq:com1}, we get
\begin{multline} \label{eq:36} 2^{j(\frac{d}{p}-\frac{d}{p*})}\norme{{\wt Z}_j(t)}_{L^p}\leq 2^{j(\frac{d}{p}-\frac{d}{p*})}\norme{{\wt Z}_{0,j}}_{L^p}+Cc_j\int_0^t \|\nabla Z^1\|_{\dot{\mathbb{B}}^{\frac{d}{p}}_{p,1}}\bigr\|\wt Z\bigr\|_{\dot{\mathbb{B}}^{\frac{d}{p}-\frac{d}{p*}}_{p,1}} \\  +\int_0^t2^{j(\frac{d}{p}-\frac{d}{p*})} \biggl\|\dot{\Delta}_j\sum_{k=1}^d\left(A^k(V^2)-A^k(V^1)\right)\d_kZ^2\biggr\|_{L^p}\cdot\qquad\qquad\end{multline}
Therefore, summing up on $j\leq J_1$, we arrive at
\begin{multline} \label{eq:37} \norme{\wt {Z}(t)}_{\dot{\mathbb{B}}^{\frac{d}{p}-\frac{d}{p*}}_{p,1}}^\ell\leq \norme{\wt{Z}_{0}}^\ell_{\dot{\mathbb{B}}^{\frac{d}{p}-\frac{d}{p*}}_{p,1}}
+C\int_0^t \|\nabla Z^1\|_{\dot{\mathbb{B}}^{\frac{d}{p}}_{p,1}}\bigr\|\wt Z\bigr\|_{\dot{\mathbb{B}}^{\frac{d}{p}-\frac{d}{p*}}_{p,1}} \\  
+C\int_0^t \biggl\|\sum_{k=1}^d\left(A^k(V^2)-A^k(V^1)\right)\d_kZ^2\biggr\|^\ell_{\dot{\mathbb{B}}^{\frac{d}{p}-\frac{d}{p*}}_{p,1}}\cdot\qquad\qquad\end{multline}

Since $A^k(V^2)-A^k(V^1)=\wt A^k(\wt Z),$ bounding 
the last term just follows from 
Inequality \eqref{eq:prod2} with $s=d/p-d/p^*,$ namely
\begin{equation}\label{eq:loi1}
\|a\,b\|_{\dot{\mathbb{B}}^{\frac{d}{p}-\frac{d}{p*}}_{p,1}}\lesssim 
\|a\|_{\dot\B^{\frac dp}_{p,1}}\,\|b\|_{\dot{\mathbb{B}}^{\frac{d}{p}-\frac{d}{p*}}_{p,1}}\end{equation}
We conclude that 
\begin{equation} \label{eq:uniqbf} \norme{\wt Z(t)}_{\dot{\mathbb{B}}^{\frac{d}{p}-\frac{d}{p*}}_{p,1}}^\ell\lesssim \norme{{\wt Z}_{0}}^\ell_{\dot{\mathbb{B}}^{\frac{d}{p}-\frac{d}{p*}}_{p,1}}+\int_0^t \|(\nabla Z^1,\nabla Z^2)\|_{\dot{\mathbb{B}}^{\frac{d}{p}}_{p,1}}\bigr\|\wt Z\bigr\|_{\dot{\mathbb{B}}^{\frac{d}{p}-\frac{d}{p*}}_{p,1}}.\end{equation}

\smallbreak\noindent{\sl Step 2: Estimates for the high frequencies.}
For all $j\in\Z,$ the function $\wt Z_j$ satisfies
$$
\d_t\wt Z_j+\sum_{k=1}^d\dot S_{j-1}A^k(V^1)\d_k\wt Z_j+L\wt Z_j
=\ddj\biggl(\sum_{k=1}^d \bigl(A^k(V^2)-A^k(V^1)\bigr)\d_kZ^2\biggr)+
\sum_{k=1}^d R_k $$
with $$\displaystyle R_k\triangleq\dot S_{j-1}
(\wt A^k(Z^1))\,\ddj \d_k\wt Z-\dot{\Delta}_j\bigl(\wt A^k(Z^1)\d_k \wt Z\bigr)\cdotp$$
Hence, performing the classical procedure, we end up with
\begin{eqnarray*}  \norme{\wt Z(t)}_{\dot{\mathbb{B}}^{\frac{d}{2}}_{2,1}}^h&\lesssim& \norme{{\wt Z}_0}^h_{\dot{\mathbb{B}}^{\frac{d}{2}}_{2,1}}+\int_0^t \|\nabla Z^1\|_{L^\infty}\bigr\|{\wt Z}\bigr\|_{\dot{\mathbb{B}}^{\frac{d}{2}}_{2,1}}^h \\&&  +\int_0^t \biggl\|\sum_{k=1}^d\left(A^k(V^2)-A^k(V^1)\right)\d_kZ^2\biggr\|^h_{\dot{\mathbb{B}}^{\frac{d}{2}}_{2,1}}+\int_0^t\sum_{j\geq J_1}2^{j\frac{d}{2}}\bigl\|\sum_{k=1}^dR_k\|_{L^2}.\end{eqnarray*}
Applying Lemma  \ref{CP} to $w=\wt A^k(Z^1)$
and $z=\d_k\wt Z$
with $s=\frac{d}{2}$, $k=1$, $\sigma_1=\frac{d}{p}+2$ and $\sigma_2=\frac{d}{p}+1,$ and 
remembering that all the maps $V\mapsto A^k(V)$ are linear, 
we get
\begin{eqnarray*}
\sum_{j\geq J_1}\biggl(2^{j\frac{d}{2}}\bigl\|\sum_{k=1}^dR_k\|_{L^2}\biggr)&\lesssim& \norme{\nabla Z^1}_{\dot{\mathbb{B}}^{\frac dp}_{p,1}}\norme{\nabla \wt Z}^h_{\dot{\mathbb{B}}^{\frac{d}{2}-1}_{2,1}}
+ \norme{\nabla\wt  Z}_{\dot{\mathbb{B}}^{\frac{d}{p}-\frac{d}{p*}}_{p,1}}^\ell\norme{Z^1}^\ell_{\dot{\mathbb{B}}^{\frac{d}{p}+2}_{p,1}}\\&&\quad\quad\quad+\norme{\nabla \wt Z}_{\dot{\mathbb{B}}^{\frac dp-1}_{p,1}}\norme{Z^1}^h_{\dot{\mathbb{B}}^{\frac{d}{2}+1}_{2,1}}
+ \norme{\nabla \wt Z}^\ell_{\dot{\mathbb{B}}^{\frac{d}{p}+1}_{p,1}}\norme{ \nabla  Z^1}^\ell_{\dot{\mathbb{B}}^{\frac{d}{p}-\frac{d}{p*}}_{p,1}}.
\end{eqnarray*}
Using \eqref{eq:prod4} with $a=\d_kZ^2,$ $b=A^k(V^2)-A^k(V^1),$
$s=\frac{d}{2}$ and $\sigma=\frac{d}{p}+1$ yields
\begin{eqnarray*}
&&\|\sum_{k=1}^d\left(A^k(V^2)-A^k(V^1)\right)\d_kZ^2\|_{\dot\B^{\frac d2}_{2,1}}^h\\&&\lesssim
\|\nabla Z^2\|_{\dot\B^{\frac dp}_{p,1}}\|\wt Z\|^h_{\dot\B^{\frac{d}{2}}_{2,1}}
+\|\wt Z\|_{\dot\B^{\frac dp}_{p,1}}\|\nabla Z^2\|^h_{\dot\B^{\frac{d}{2}}_{2,1}}+ \|\nabla Z^2\|^\ell_{\dot\B^{\frac dp-\frac{d}{p*}}_{p,1}}\|\wt Z\|^\ell_{\dot\B^{\frac{d}{p}+1}_{p,1}}
+ \|\wt Z\|^\ell_{\dot\B^{\frac dp-\frac{d}{p*}}_{p,1}}\|\nabla Z^2\|^\ell_{\dot\B^{\frac{d}{p}+1}_{p,1}}.
\end{eqnarray*}
Gathering the above estimates
and using once more the embedding
$\dot\B^{\frac dp}_{p,1}\hookrightarrow L^\infty,$
we obtain
\begin{eqnarray} \label{eq:uniqhf} \notag \norme{\wt Z(t)}_{\dot{\mathbb{B}}^{\frac{d}{2}}_{2,1}}^h&\lesssim& \norme{{\wt Z}_0}^h_{\dot{\mathbb{B}}^{\frac{d}{2}}_{2,1}}+\int_0^t \|(\nabla Z^1,\nabla Z^2)\|_{\dot{\mathbb{B}}^{\frac{d}{p}}_{p,1}}\|{\wt Z}\bigr\|_{\dot{\mathbb{B}}^{\frac{d}{2}}_{2,1}}^h
+\int_0^t \|(\nabla Z^1,\nabla Z^2)\|^h_{\dot{\mathbb{B}}^{\frac{d}{2}}_{2,1}}\|{\wt Z}\bigr\|_{\dot{\mathbb{B}}^{\frac{d}{p}}_{p,1}}\\&&+ \int_0^t\| (Z^1,Z^2)\|^\ell_{\dot\B^{\frac dp}_{p,1}}\|\wt Z\|^\ell_{\dot\B^{\frac{d}{p}+1}_{p,1}}
+\int_0^t \|\wt Z\|^\ell_{\dot\B^{\frac dp-\frac{d}{p*}}_{p,1}}\|(Z^1, Z^2)\|^\ell_{\dot\B^{\frac{d}{p}+2}_{p,1}}.
\end{eqnarray}

\smallbreak\noindent{\sl Step 3: Conclusion.}
Summing \eqref{eq:uniqbf} and \eqref{eq:uniqhf} together, we get for all $t\geq0,$
$$\begin{aligned} \norme{\wt Z(t)}^\ell_{\dot{\mathbb{B}}^{\frac{d}{p}-\frac{d}{p*}}_{p,1}}
&+\norme{\wt Z(t)}_{\dot{\mathbb{B}}^{\frac{d}{2}}_{2,1}}^h\lesssim \norme{\wt Z_{0}}^\ell_{\dot{\mathbb{B}}^{\frac{d}{p}-\frac{d}{p*}}_{p,1}}+\norme{\wt Z_0}^h_{\dot{\mathbb{B}}^{\frac{d}{2}}_{2,1}}\\&+\int_0^t \|(\nabla Z^1,\nabla Z^2)\|_{\dot{\mathbb{B}}^{\frac{d}{p}}_{p,1}}(\|{\wt Z}\bigr\|_{\dot{\mathbb{B}}^{\frac{d}{p}-\frac{d}{p*}}_{p,1}}^\ell+\|{\wt Z}\bigr\|_{\dot{\mathbb{B}}^{\frac{d}{2}}_{2,1}}^h)
+\int_0^t \|(\nabla Z^1,\nabla Z^2)\|^h_{\dot{\mathbb{B}}^{\frac{d}{2}}_{2,1}}\|{\wt Z}\bigr\|_{\dot{\mathbb{B}}^{\frac{d}{p}}_{p,1}}\\&+ \int_0^t\| (Z^1,Z^2)\|^\ell_{\dot\B^{\frac dp}_{p,1}}\|\wt Z\|^\ell_{\dot\B^{\frac{d}{p}+1}_{p,1}}+\int_0^t \|\wt Z\|^\ell_{\dot\B^{\frac dp-\frac{d}{p*}}_{p,1}}\|(Z^1, Z^2)\|^\ell_{\dot\B^{\frac{d}{p}+2}_{p,1}},\end{aligned}$$
which, by virtue of    \eqref{eq:bernstein}  yields
the desired estimate of Lemma \ref{lemma:uniq}.\end{proof}

In order to prove the uniqueness part of Theorem \ref{ThmExistLp},
consider two solutions $Z^1$ and $Z^2$ of \eqref{GE} (not necessarily small)
 in  the space $E_p,$ that correspond  to the same initial data $Z_0$. 
 Then, the result follows from  Lemma \ref{lemma:uniq},
\emph{provided we prove that the difference between the two 
solutions belongs to $F_p(T)$ for all 
$T>0.$}


Just denoting by $Z$ one of those two solutions, we have \begin{equation}
  \d_tZ=-\sum_{k=1}^d A^k(V)\d_kZ-LZ\cdotp\end{equation}
By interpolation in Besov spaces and H\"older inequality with respect to the time variable, 
since $Z^\ell$ is in $L^\infty(\R^+;\dot\B^{\frac dp}_{p,1})\cap L^1(\R^+;\dot\B^{\frac dp+2}_{p,1}),$  we get
\begin{equation}\label{eq:Lr}\nabla Z^\ell\in L^r(\R^+;\dot\B^{\frac dp-\frac d{p*}}_{p,1})
\with \frac1r\triangleq\frac12-\frac d4+\frac d{2p}\cdotp\end{equation}
The same property holds for  $Z^h$  since it belongs  to $ L^1(\dot\B^{\frac dp+1}_{p,1})\cap L^\infty(\dot\B^{\frac dp+1}_{p,1}).$
We also know that $A^k(V)-\bar A^k$ is in  $L^\infty(\R^+;\dot\B^{\frac dp}_{p,1}).$ Therefore, from 
the product laws in Besov spaces that have been recalled in Proposition \ref{LP}, we
have that  $\partial_tZ_1$ is in  $L^r(\R^+;\dot\B^{\frac dp-\frac{d}{p*}}_{p,1}),$  and thus
\begin{equation}\label{eq:uniq1} Z_1-Z_{1,0}\in \cC^{\frac1{r'}}_{loc}(\R^+;\dot\B^{\frac dp -\frac{d}{p*}}_{p,1}).\end{equation}
We conclude that $Z_1-Z_{1,0}$ is in $F_p(T)$ for all finite $T.$

Owing to the $0$-th order term $LZ$ in the equation, 
in order to justify that $(Z_2-Z_{2,0})\in F_p(T),$ we have 
to proceed slightly differently. 
Now, we notice that 
$$\d_t(e^{tL_2}Z_2) = -e^{tL_2}\biggl(\sum_{k=1}^d
A^k_{2,1}(V)\d_kZ_1+A_{2,2}^k(V)\d_kZ_2\biggr)\cdotp$$
Arguing as above, we see that the right-hand side
is in  $L^r(\R^+;\dot\B^{\frac dp-\frac{d}{p*}}_{p,1}),$
which, as before, allows to conclude 
that $Z_2-Z_{2,0}\in \cC^{\frac1{r'}}_{loc}(\R^+;\dot\B^{\frac dp -\frac{d}{p*}}_{p,1}).$

Back to our two solutions $Z_1$ and $Z_2,$ since they coincide initially, the above arguments ensure that $Z_1-Z_2$ is in $F_p(T).$
Hence, combining Lemma \ref{lemma:uniq}, Gronwall lemma 
and the fact that the low frequencies of $Z^1$ and $Z^2$ 
(resp. the high frequencies of 
$\nabla Z^1$ and $\nabla Z^2$) are bounded in $L^1(0,T;\dot\B^{\frac dp}_{p,1})$ (resp. in $L^1(0,T;\dot\B^{\frac d2}_{2,1})$)
for all $T>0$ completes the proof of  uniqueness. 


\section{Relaxation limit for the compressible Euler system}\label{s:lim}
In this section we prove Theorem \ref{Thm-relax}. We shall often use  that, as a consequence of \eqref{eq:rescale2}, \eqref{eq:comparaison} and of the definition of $J_\varepsilon$, 
there exists $C>0$ such that  for all $s\in\R$ and $\varepsilon>0,$
\begin{equation}\label{eq:bernep}
\|f\|^{\ell,J_\varepsilon}_{\dot\B^{s}_{p,1}}\leq \frac{C}{\varepsilon}\|f\|^{\ell,J_\varepsilon}_{\dot\B^{s-1}_{p,1}} \andf  \|f\|^{h,J_\varepsilon}_{\dot \B^{s}_{2,1}}\leq C\varepsilon\|f\|^{h,J_\varepsilon}_{\dot\B^{s+1}_{2,1}}.\end{equation}

\subsection{Reformulation of the problem and derivation of 
the limit system}

Let $(c,v)$ be a  solution from Theorem \ref{ThmExistLpCED}.
As in \cite{CoulombelLin,CoulombelGoudon}, we 
perform the following `diffusive'  rescaling: 
$$(\tilde{c}^\varepsilon,\tilde{v}^\varepsilon)(\tau,x)
\triangleq(c,\frac{v}{\varepsilon})(t,x)
\with \tau=\varepsilon t.$$ 
The couple $(\widetilde{c}^\varepsilon,\widetilde{v}^\varepsilon)$ satisfies: 
\begin{equation} \displaystyle\left\{\displaystyle \begin{matrix}\partial_t\tilde{c}^\varepsilon+\tilde{v}^\varepsilon\cdot\nabla \tilde{c}^\varepsilon+\tilde{\gamma}\tilde{c}^\varepsilon\div \tilde{v}^\varepsilon=0,\\[1ex]\displaystyle \varepsilon^2\left(\partial_t\tilde{v}^\varepsilon+\tilde{v}^\varepsilon\cdot \nabla \tilde{v}^\varepsilon\right)+\check{\gamma}\wt c^\varepsilon\nabla \wt c^\varepsilon+\tilde{v}^\varepsilon
=0. \end{matrix} \right.
\label{CEDRelax0}\end{equation}
As a consequence of Theorem \ref{ThmExistLpCED} and of 
\eqref{eq:bernep}, we readily get the following 
uniform estimate\footnote{ The crucial bound on  $\norme{\widetilde{c}^\varepsilon-\bar{c}}_{L^2(\dot{\mathbb{B}}^{\frac{d}{p}+1}_{p,1})}$ can be easily deduced from the other bounds.}  which will be a key ingredient in our study of the relaxation limit:
\begin{eqnarray}&&\norme{\widetilde{c}^\varepsilon-\bar{c}}^{\ell,J_\varepsilon}_{L^\infty(\dot{\mathbb{B}}^{\frac{d}{p}}_{p,1})}+\varepsilon\norme{\wt v^\varepsilon}^{\ell,J_\varepsilon}_{L^\infty(\dot{\mathbb{B}}^{\frac{d}{p}}_{p,1})}+\varepsilon\norme{\widetilde{c}^\varepsilon-\bar{c}}^{h,J_\varepsilon}_{L^\infty(\dot{\mathbb{B}}^{\frac{d}{2}+1}_{2,1})}+\varepsilon^2\norme{\wt v^\varepsilon}^{h,J_\varepsilon}_{L^\infty(\dot{\mathbb{B}}^{\frac{d}{2}+1}_{2,1})}
+\norme{\widetilde{c}^\varepsilon-\bar{c}}^{\ell,J_\varepsilon}_{L^1(\dot{\mathbb{B}}^{\frac{d}{p}+2}_{p,1})}\nonumber\\&&+\frac{1}{\varepsilon}\norme{\widetilde{c}^\varepsilon-\bar{c}}^{h,J_\varepsilon}_{L^1(\dot{\mathbb{B}}^{\frac{d}{2}+1}_{2,1})}
+\norme{\wt v^\varepsilon}^{h,J_\varepsilon}_{L^1(\dot{\mathbb{B}}^{\frac{d}{2}+1}_{2,1})}+\norme{\wt v^\varepsilon}^{\ell,J_\varepsilon}_{L^1(\dot{\mathbb{B}}^{\frac{d}{p}+1}_{p,1})}+\norme{\widetilde{c}^\varepsilon-\bar{c}}_{L^2(\dot{\mathbb{B}}^{\frac{d}{p}+1}_{p,1})}+\norme{\wt v^\varepsilon}_{L^2(\dot{\mathbb{B}}^{\frac{d}{p}}_{p,1})}\nonumber\\&&\qquad\qquad\qquad\qquad\qquad\qquad\qquad\qquad\qquad\qquad\qquad\qquad\qquad\qquad+\frac{1}{\varepsilon}\norme{\wt W^\varepsilon}_{L^1(\dot{\mathbb{B}}^{\frac{d}{p}}_{p,1})}\leq Cc_0 \label{UniformTilde}
\end{eqnarray}
where $\wt W^\varepsilon=\check{\gamma}\wt c^\varepsilon\nabla \wt c^\varepsilon+\tilde v^\varepsilon$ and 
$J_\varepsilon=-\lfloor\log_2(\varepsilon)\rfloor+k_p$ for some
$k_p\in\Z.$
\bigbreak
Let us define  the  density $\wt\rho^\varepsilon$ and reference density 
$\bar\rho$ from  \eqref{eq:c}. Then, $(\wt\rho^\varepsilon,\wt v^\varepsilon)$ obeys the following system:
 \begin{equation} \displaystyle\left\{\displaystyle \begin{matrix}\partial_t\widetilde{\rho}^\varepsilon+\operatorname{div}\, (\widetilde{\rho}^\varepsilon\widetilde{v}^\varepsilon)=0,\\[1ex]\displaystyle \varepsilon^2\left(\partial_t\tilde{v}^\varepsilon+\tilde{v}^\varepsilon\cdot \nabla \tilde{v}^\varepsilon\right)+\frac{\nabla P(\wt\rho^\varepsilon)}{\wt\rho^\varepsilon}+\tilde{v}^\varepsilon
=0. \end{matrix} \right.
\label{CEDRelax}\end{equation}
Owing to \eqref{UniformTilde}, $\varepsilon \wt v^\varepsilon$ and $\nabla \wt v^\varepsilon$ are uniformly bounded in the spaces
$L^\infty(\R^+;\dot\B^{\frac dp}_{p,1})$ and $L^1(\R^+;\dot\B^{\frac dp}_{p,1}),$ respectively. This implies  that
$$\varepsilon^2 \wt v^\varepsilon\cdot\nabla \wt v^\varepsilon=\cO(\varepsilon)\quad\hbox{in }\ L^1(\R^+;\dot\B^{\frac dp}_{p,1}).$$
The uniform estimate \eqref{UniformTilde} also implies that $\varepsilon^2\d_t\wt v^\varepsilon$ tends to $0$ in the sense of distributions.
Plugging this information in the second equation of \eqref{CEDRelax}, one may conclude  that
  \begin{equation}\label{eq:weakW}\tilde{v}^\varepsilon+\frac{\nabla P(\wt\rho^\varepsilon)}{\widetilde{\rho}^\varepsilon}\rightharpoonup 0 \ \text{  in  }\  \mathcal{D}'(\mathbb{R}^+\times\mathbb{R}^d).
  \end{equation}
  Let us remember  that
\begin{equation}\label{eq:relation} 
\wt\rho^\varepsilon -\bar \rho= \biggl(\frac{\gamma-1}{\sqrt{4A\gamma}}\:\wt c^\varepsilon\biggr)^{\frac2{\gamma-1}}
- \biggl(\frac{\gamma-1}{\sqrt{4A\gamma}}\:\bar c\biggr)^{\frac2{\gamma-1}}\cdotp\end{equation}
{}From  \eqref{eq:bernep} and \eqref{UniformTilde},
it is easy to see that 
 $$\norme{\wt c^\varepsilon-\bar{c}}_{L^\infty(\dot{\mathbb{B}}^{\frac{d}{p}}_{p,1})}+\norme{\wt c^\varepsilon-\bar{c}}_{L^2(\dot{\mathbb{B}}^{\frac{d}{p}+1}_{p,1})}\leq c_0. $$
 Hence, using Proposition \ref{Composition} and  \eqref{eq:relation}  gives
  \begin{equation}\label{eq:rhowt}\norme{\wt\rho^\varepsilon-\bar{\rho}}_{L^\infty(\dot{\mathbb{B}}^{\frac{d}{p}}_{p,1})}+\norme{\wt\rho^\varepsilon-\bar{\rho}}_{L^2(\dot{\mathbb{B}}^{\frac{d}{p}+1}_{p,1})}\leq c_0. \end{equation}
In particular
$\widetilde{\rho}^\varepsilon-\bar{\rho}$ is uniformly bounded in $L^\infty(\R^+;\dot \B^{\frac dp}_{p,1}).$ Therefore, there exists $\cN$ in  $\bar\rho+L^\infty(\R^+;\dot \B^{\frac dp}_{p,1})$
such that, up to subsequence, 
 \begin{equation}\label{eq:weakn}\wt{\rho}^\varepsilon-\bar{\rho} \overset{\ast}{\rightharpoonup} \mathcal{N}-\bar{\rho}\ \text{ in  }\  L^\infty(\mathbb{R}^+;\dot{\mathbb{B}}^{\frac{d}{p}}_{p,1}).\end{equation}
Now, observing that 
\begin{equation}\label{eq:Wepsilon}
\wt\rho^\varepsilon \wt W^\varepsilon=\nabla P(\wt\rho^\varepsilon)+ {\widetilde{\rho}^\varepsilon}\tilde v^\varepsilon,
\end{equation} the first equation of \eqref{CEDRelax} may be rewritten 
\begin{equation}\label{eq:safd}
\partial_t\wt\rho^\varepsilon-\Delta P(\wt\rho^\varepsilon)=
\wt S^\varepsilon\with \wt S^\varepsilon=-\div(\widetilde{\rho}^\varepsilon \wt W^\varepsilon).
\end{equation}
Hence, combining \eqref{eq:weakW}, \eqref{eq:weakn} and \eqref{eq:Wepsilon}, 
it can be anticipated that $\cN$ satisfies
\begin{equation}\label{CauchyPbM}\partial_t\mathcal{N}-\Delta P(\mathcal{N}) =0.\end{equation}


\subsection{Proving the strong convergence to the limit system}

Having determined the limit system, we  are now going to use the uniform estimate
\eqref{UniformTilde} to prove the strong convergence of the density 
to a solution of \eqref{CauchyPbM}, with an explicit rate of convergence.
As a first, let us remember that \eqref{UniformTilde}
and \eqref{eq:rhowt} imply 
that\footnote{Unless $\gamma=3,$  
we do not know how to deduce specific information 
on the low (resp. high) frequencies 
of  $\rho-\bar{\rho}$ from that of  $c-\bar c.$
This is due to the \emph{nonlinear} relation
between these two functions. } 
\begin{equation}\label{eq:boundsrho}
\norme{\wt \rho^\varepsilon-\bar{\rho}}_{L^\infty(\dot{\mathbb{B}}^{\frac{d}{p}}_{p,1})}+\norme{\wt\rho^\varepsilon-\bar{\rho}}_{L^2(\dot{\mathbb{B}}^{\frac{d}{p}+1}_{p,1})}+\varepsilon^{-1}
\norme{\wt W^\varepsilon}_{L^1(\dot{\mathbb{B}}^{\frac{d}{p}}_{p,1})}\leq c_0. \end{equation}

Next, setting  $\bar{\mathcal{N}}=\bar\rho,$ 
Proposition  \ref{PropExistMAppendix} below guarantees 
that Equation \eqref{CauchyPbM} supplemented with any initial data
$\mathcal{N}_0$ such that $\mathcal{N}_0-\bar{\mathcal{N}}\in\dot\B^{\frac dp}_{p,1}$ is small enough 
admits  a unique global solution $\mathcal{N}$ such that  $\mathcal{N}-\bar{\mathcal{N}}\in \cC_b(\R^+;\dot \B^{\frac dp}_{p,1}) \cap L^1(\R^+;\dot\B^{\frac dp+2}_{p,1}).$
\medbreak
We can now prove Theorem \ref{Thm-relax}, assuming that  $\varepsilon>0$ is small, and that
\begin{equation}\label{eq:diff}\|\widetilde{\rho}_0^\varepsilon-\mathcal{N}_0\|_{\dot\B^{\frac{d}{p}-1}_{p,1}}\leq \varepsilon.\end{equation}

To derive the convergence rate we will estimate the difference of the  solutions to the following two equations:
\begin{equation}
\partial_t\mathcal{N}-\Delta P(\mathcal{N}) =0
\end{equation}
and
\begin{equation} \partial_t\tilde{\rho}^\varepsilon+\operatorname{div}\, (\widetilde{\rho}^\varepsilon\tilde{v}^\varepsilon)=0. \label{eq-n}\end{equation}
Recall that \eqref{eq-n} may be rewritten in terms of the damped mode $\widetilde{W}^\e$ as in \eqref{eq:safd}. 
Hence $\dD^{\varepsilon}\triangleq\widetilde{\rho}^{\varepsilon}-\mathcal{N}$  satisfies
$$\partial_t \dD^{\varepsilon}-\Delta(P(\widetilde\rho^\varepsilon)-P(\mathcal{N}))=\wt S^{\varepsilon}.$$
In light of Taylor formula, there exists a smooth function $H_1$ vanishing at $\bar\rho=\bar{\mathcal{N}}$ such that 
 $$ P(\widetilde\rho^\varepsilon)-P(\bar{\rho})= P'(\bar{\rho})\,(\widetilde\rho^\varepsilon-\bar{\rho})+ H_1(\widetilde\rho^\varepsilon)\,(\widetilde\rho^\varepsilon-\bar{\rho}).
 $$
 and $$P(\mathcal{N})-P(\bar{\mathcal{N}})= P'(\bar{\mathcal{N}})\,(\mathcal{N}-\bar{\mathcal{N}}) + H_1(\mathcal{N})\,(\mathcal{N}-\bar{\mathcal{N}}).$$
  Hence  we have
\begin{align*}\partial_t\dD^{\varepsilon}-P'(\bar\rho)\Delta\dD^\varepsilon=\Delta \left(\dD^\varepsilon\: H_1(\widetilde\rho^\varepsilon)\right)+\Delta \left((H_1(\widetilde\rho^\varepsilon)-H_1(\mathcal{N}))\mathcal{N}\right)+\wt S^{\varepsilon}.
\end{align*}
Then, using endpoint maximal regularity estimates for the heat equation
(see e.g. \cite{HJR}) yields
 \begin{multline}\label{finalrelax}\|\dD^{\varepsilon}\|_{L^\infty(\dot \B^{\frac{d}{p}-1}_{p,1})}+\|\dD^{\varepsilon}\|_{L^1(\dot \B^{\frac{d}{p}+1}_{p,1})}\lesssim \|\dD_0^{\varepsilon}\|_{\dot\B^{\frac{d}{p}-1}_{p,1}}+\|\wt S^{\varepsilon}\|_{L^1(\dot \B^{\frac{d}{p}-1}_{p,1})}\hfill\cr\hfill\|\dD^\varepsilon\: (H_1(\widetilde\rho^\varepsilon)-H_1(\bar{\rho}))\|_{L^1(\dot\B^{\frac{d}{p}+1}_{p,1})}+\|(H_1(\widetilde\rho^\varepsilon)-H_1(\mathcal{N}))(\mathcal{N}-\bar{\mathcal{N}})\|_{L^1(\dot\B^{\frac{d}{p}+1}_{p,1})}.
 \end{multline}
 Basic product laws  give us:
$$\begin{aligned}\|\wt S^{\varepsilon}\|_{L^1(\dot{\B}^{\frac{d}{p}-1}_{p,1})}\lesssim \|\wt{\rho}^\varepsilon \wt W^\varepsilon\|_{L^1(\dot{\B}^{\frac{d}{p}}_{p,1})}\lesssim
\|\wt W^{\varepsilon}\|_{L^1(\dot{\B}^{\frac{d}{p}}_{p,1})}\Bigl(\bar{\rho}+\|\wt \rho^{\varepsilon}-\bar\rho\|_{L^\infty(\dot{\B}^{\frac{d}{p}}_{p,1})}\Bigr)\cdotp
\end{aligned}$$
Hence, taking advantage of Inequality \eqref{eq:boundsrho},  we get
\begin{equation} \label{Est:Source}\|\wt S^{\varepsilon}\|_{L^1(\dot{\B}^{\frac{d}{p}-1}_{p,1})} \leq C\varepsilon. \end{equation}
Propositions \ref{LP} and \ref{Composition} give us
$$\begin{aligned}\|\dD^\varepsilon\: (H_1(\widetilde\rho^\varepsilon)-H_1(\bar{\rho}))\|_{L^1(\dot\B^{\frac{d}{p}+1}_{p,1})}&\lesssim \|\dD^\varepsilon\|_{L^1(\dot \B^{\frac{d}{p}+1}_{p,1})}\|(\widetilde\rho^\varepsilon-\bar{\rho},\mathcal{N}-\bar{\mathcal{N}})\|_{L^\infty(\dot \B^{\frac{d}{p}}_{p,1})},\\
\|(H_1(\widetilde\rho^\varepsilon)-H_1(\mathcal{N}))(\mathcal{N}-\bar{\mathcal{N}})\|_{L^1(\dot\B^{\frac{d}{p}+1}_{p,1})}
&\lesssim
\|\dD^\varepsilon\|_{L^2(\dot \B^{\frac{d}{p}}_{p,1})}\|(\widetilde\rho^\varepsilon-\bar\rho,\mathcal{N}-\bar{\mathcal{N}})\|_{L^2(\dot \B^{\frac{d}{p}+1}_{p,1})}.
\end{aligned}$$
Thanks to Inequality  \eqref{eq:boundsrho} and Proposition \ref{PropExistMAppendix}, we have $$\|(\wt{\rho}^\varepsilon-\bar\rho,\mathcal{N}-\bar{\mathcal{N}})\|_{L^\infty(\mathbb{B}^{\frac{d}{p}}_{p,1})}+\|(\wt{\rho}^\varepsilon-\bar\rho,\mathcal{N}-\bar{\mathcal{N}})\|_{L^2(\mathbb{B}^{\frac{d}{p}+1}_{p,1})} \leq c_0\ll1.$$ 
Hence, reverting to  \eqref{finalrelax} yields
$$\|\dD^{\varepsilon}\|_{L^\infty(\dot \B^{\frac{d}{p}-1}_{p,1})}+\|\dD^{\varepsilon}\|_{L^1(\dot \B^{\frac{d}{p}+1}_{p,1})}\lesssim 
\|\dD_0^\varepsilon\|_{\dot \B^{\frac{d}{p}-1}_{p,1}}+\varepsilon, $$
which concludes the proof of Theorem \ref{Thm-relax}.


\section{Appendix}

Here we gather a few technical results that have been used repeatedly in the paper. 
We often used the following  well known result  (see  e.g. \cite{CBD1} for the proof). 
\begin{Lemme}\label{SimpliCarre}
Let  $p\geq 1$ and $X : [0,T]\to \mathbb{R}^+$ be a continuous function such that $X^p$ is a.e. differentiable. We assume that there exist  a constant $b\geq 0$ and  a measurable function $A : [0,T]\to \mathbb{R}^+$ 
such that 
 $$\frac{1}{p}\frac{d}{dt}X^p+bX^p\leq AX^{p-1}\quad\hbox{a.e.  on }\ [0,T].$$ 
 Then, for all $t\in[0,T],$ we have
$$X(t)+b\int_0^tX\leq X_0+\int_0^tA.$$
\end{Lemme}

When proving Theorem \ref{Thm-relax}, we used the following 
global  existence result for \eqref{CauchyPbM}.
\begin{Prop}\label{PropExistMAppendix}
Let $1\leq p<\infty$ and $\mathcal{N}_0-\bar{\mathcal{N}}\in\dot\B^{\frac dp}_{p,1}$
with $\bar{\mathcal N}>0$. There exists a constant $c_0>0$ such that if 
\begin{equation}\label{eq:smallN0}\|\mathcal{N}_0-\bar{\mathcal{N}}\|_{\dot\B^{\frac dp}_{p,1}}\leq c_0 \end{equation}
then, System \eqref{CauchyPbM} 
with  a pressure function $P$ satisfying \eqref{Pression1} 
and supplemented with initial data $\mathcal{N}_0$
has a unique global solution $\mathcal{N}$ such that  $\mathcal{N}-\bar{\mathcal{N}}\in \cC_b(\R^+;\dot \B^{\frac dp}_{p,1})\cap L^1(\R^+;\dot\B^{\frac dp+2}_{p,1}).$
\end{Prop}
\begin{proof}
Assume that we have a smooth solution $\mathcal{N}$ of \eqref{CauchyPbM}.
There exists a function $H_1$ vanishing at $\bar{\mathcal{N}}$ such that:
$$P(\mathcal{N})-P(\bar{\mathcal{N}})= P'(\bar{\mathcal{N}})\,(\mathcal{N}-\bar{\mathcal{N}}) + H_1(\mathcal{N})\,(\mathcal{N}-\bar{\mathcal{N}}).$$
Therefore one can rewrite  \eqref{CauchyPbM} as
$$\partial_t\mathcal{N}-P'(\bar{\mathcal{N}})\Delta \mathcal{N}=\Delta\Bigl(H_1(\mathcal{N})\,(\mathcal{N}-\bar{\mathcal{N}})\Bigr).$$
Hence, using classical endpoint maximal regularity estimates for the 
heat equation (see e.g. \cite{HJR}), we get for all $T>0,$
\begin{equation}\label{eq:NN}\|\mathcal{N}-\bar{\mathcal{N}}\|_{L^\infty_T(\dot\B^{\frac dp}_{p,1})}
+\|\mathcal{N}-\bar{\mathcal{N}}\|_{L^1_T(\dot\B^{\frac dp+2}_{p,1})}
\lesssim \|\mathcal{N}_0-\bar{\mathcal{N}}\|_{\dot\B^{\frac dp}_{p,1}}
+\|H_1(\mathcal{N})\,(\mathcal{N}-\bar{\mathcal{N}})\|_{L^1_T(\dot\B^{\frac dp+2}_{p,1})}. 
\end{equation}
Combining product laws from \eqref{eq:prod1} with 
composition estimates of Proposition \ref{Composition} yields
$$\|H_1(\mathcal{N})\,(\mathcal{N}-\bar{\mathcal{N}})\|_{L^1_T(\dot\B^{\frac dp+2}_{p,1})}\lesssim\|\mathcal{N}-\bar{\mathcal{N}}\|_{L^\infty_T(\dot\B^{\frac dp}_{p,1})}
\|\mathcal{N}-\bar{\mathcal{N}}\|_{L^1_T(\dot\B^{\frac dp+2}_{p,1})}.$$
Hence the left-hand side of \eqref{eq:NN} may be bounded
for all $T>0$ in terms of the data provided \eqref{eq:smallN0} is 
satisfied with a small enough $c_0.$ From that point, 
it is easy  to work out a fixed point procedure
yielding  the global existence of a solution for \eqref{CauchyPbM}.
Uniqueness follows from similar estimates. 
\end{proof}

 The first part of the existence proof relied
  on the following classical local well-posedness result
  for hyperbolic symmetric systems. 
 \begin{Thm}{\cite[Chap.~4]{HJR}}  \label{ThmExistLocalLp}
Consider the following hyperbolic system:
 $$\left\{\begin{array}{l} \d_t U+\sum_{k=1}^d A_k(U)\d_kU+A_0(U)=0,\\[1ex]
U|_{t=0}=U_0,\end{array}\right.\leqno(QS)$$
where $A_k,$ $k=0,\cdots,d,$  are smooth functions from $\R^n$ to the space of $n\times n$  matrices,  that are symmetric if $k\not=0,$ 
supplemented with initial data  $U_0$  in  the nonhomogeneous Besov space $\mathbb{B}^{\frac{d}{2}+1}_{2,1}(\R^d;\R^n)$. 

Then, $(QS)$ admits a unique maximal solution $U$  in $\mathcal{C}([0,T^*[;{\mathbb{B}^{\frac{d}{2}+1}_{2,1}})\cap\mathcal{C}^1([0,T^*[;{\mathbb{B}^{\frac{d}{2}}_{2,1}}),$ and there exists a positive constant $c$ such that 
$$T^*\geq\frac{c}{\norme{U_0}_{\mathbb{B}^{\frac{d}{2}+1}_{2,1}}}\cdotp$$
Furthermore, 
$$T^*<\infty \Longrightarrow \int_0^{T^*}\norme{\nabla U}_{L^\infty}=\infty.$$
\end{Thm}

 Next, let us prove the equivalence between Condition (SK)  and the strong ellipticity condition
 for System \eqref{GE2}  pointed out in the introduction.
 \begin{Lemme}\label{l:parabolic}
Assume that $\bar A_{1,1}^k=0$  for all $k\in\{1,\cdots,d\}.$
Then, the following assertions are equivalent: \begin{itemize}
    \item System \eqref{GEQSYM} satisfies the condition (SK) at $\bar{V}$;\smallbreak
    \item $\hbox{the operator } \cA\triangleq-\sum_{k=1}^d\sum_{\ell=1}^d \bar A_{1,2}^kL_2^{-1}\bar A^{\ell}_{2,1}\d_k\d_\ell \text{ is strongly elliptic.}$
    \end{itemize}
    If one of the above assertions is satisfied and  if $\text{Supp}(\mathcal{F}Z_1)\subset\{\xi\in\mathbb{R}^d : R_1\lambda\leq|\xi|\leq R_2\lambda\}$ for some $0<R_1<R_2$ then, for all $p\in[2,\infty[,$  there exists $c=c(p,d,R_1,R_2)>0$ such that 
\begin{equation}\label{Elliptic}\int_{\R^d}\sum_{j=1}^{n_1}\sum_{k=1}^d\sum_{\ell=1}^d \bar A_{1,2}^kL_2^{-1}\bar A^{\ell}_{2,1}\d_k\d_\ell Z_1^j\: |Z_1|^{p-2}Z_1^j\geq  c \lambda^2\| Z_1\|_{L^p}^p.
\end{equation}
\end{Lemme}
\begin{proof}
 The direct implication was proved in \cite{Yong,WasiolekPeng,ThesisWasiolek}.
 For the converse implication, let us still denote by $L_2$
 the $n_2\times n_2$ (invertible) matrix of $L_2$ and set
 $$\cA_{\ell,m}(\xi)\triangleq\sum_{k=1}^d\bar A^k_{\ell,m}\xi_k,\quad
 1\leq \ell,m\leq2.$$
 Our assumptions ensure that  the symmetric parts of $L_2$ and  of the matrix $\cA_{1,2}(\xi)L_2^{-1}\cA_{2,1}(\xi)$ for all 
 $\xi\not=0$ are   positive definite. This in particular implies that
 the ranks of $\cA_{1,2}(\xi)$ and $\cA_{2,1}(\xi)$ must be equal to 
 $n_1$ and thus, so does the rank of $L_2\cA_{2,1}(\xi).$
 Now, the matrices of $L$ and of $L\cA(\xi)$ can be written by blocks as follows:
 $$
 L=\begin{pmatrix}0&0\\0&L_2\end{pmatrix}\andf
 L\cA(\xi)=\begin{pmatrix}0&0\\L_2\cA_{2,1}(\xi)& L_2\cA_{2,2}(\xi)\end{pmatrix}\cdotp
$$ 
Hence the rank of $\begin{pmatrix} L\\ L\cA(\xi)\end{pmatrix}$
is $n_1+n_2=n,$ and Condition (SK) is thus satisfied.
\medbreak
To prove Inequality \eqref{Elliptic}, we first observe that 
$L_2^{-1}$ may be replaced by its symmetric part (this leaves the left-hand side unchanged).
Then, performing an appropriate change of orthonormal basis reduces the proof
to the case where the matrix $\sum_{k=1}^d\sum_{\ell=1}^dA_{1,2}^kL_2^{-1}\bar A^{\ell}_{2,1}$
is diagonal and positive definite. From this point, 
one can argue exactly as 
in the proof of Lemma A.5 in \cite{Da01}.
\end{proof}

 The proof of the following  inequality may be found in e.g. \cite[Chap. 2]{HJR}.
\begin{Lemme} \label{Commutateur1}
There exists a constant $C$ such that  for all
$1\leq p,q,r\leq\infty$ such that $\frac{1}{p}+\frac{1}{q}=\frac{1}{r},$ 
all functions $a$ with gradient  in $L^p,$ and $b$ in $L^q,$ 
we have
$$\norme{[\dot{\Delta}_j,a]b}_{L^r}\leq C2^{-j}\norme{\nabla a}_{L^q}\norme{b}_{L^p}
\quad\hbox{for all }\ j\in\Z.$$
\end{Lemme}
The following result  is  proved in  e.g.  \cite[Chap. 2]{HJR}.  
\begin{Prop}\label{C1}    
For all $1\leq p\leq\infty$
and  $-\min(d/p,d/p')<s\leq d/p,$ we have
 \begin{equation}\label{eq:com1}
2^{js}\norme{[w,\dot{\Delta}_j]\nabla v}_{L^p}\leq Cc_j\norme{\nabla w}_{\dot{\mathbb{B}}^{\frac{d}{p}}_{p,1}}\norme{v}_{\dot{\mathbb{B}}^{s}_{p,1}}  \with\sum_{j\in\mathbb{Z}}c_j=1.
\end{equation}
 \end{Prop}
 
 The following product laws in Besov spaces have been used several times.
 \begin{Prop} \label{LP} Let $(s,p,r)$ be in $]0,\infty[\times[1,\infty]^2.$ Then, 
 $\dot{\mathbb{B}}^{s}_{p,r}\cap L^\infty$ is an algebra and we have
\begin{equation}\label{eq:prod1}
\norme{ab}_{\dot{\mathbb{B}}^{s}_{p,r}}\leq C\bigl(\norme{a}_{L^\infty}\norme{b}_{\dot{\mathbb{B}}^{s}_{p,r}}+\norme{a}_{\dot{\mathbb{B}}^{s}_{p,r}}\norme{b}_{L^\infty}\bigr)\cdotp
\end{equation}
If, furthermore, $-\min(d/p,d/p')<s\leq d/p,$ then the following inequality holds:
\begin{equation}\label{eq:prod2}
\|ab\|_{\dot\B^{s}_{p,1}}\leq C\|a\|_{\dot\B^{\frac dp}_{p,1}}\|b\|_{\dot\B^{s}_{p,1}}.
\end{equation}
Finally,  if  $-d/p<\sigma_1\leq \min(d/p,d/p')$, then the following inequality holds true: 
\begin{equation}\label{eq:prod3} 
\norme{ab}_{\dot{\mathbb{B}}^{-\sigma_1}_{p,\infty}}\leq C  \norme{a}_{\dot{\mathbb{B}}^{\frac{d}{p}}_{p,1}}\norme{b}_{\dot{\mathbb{B}}^{-\sigma_1}_{p,\infty}}.
\end{equation}
\end{Prop}
The following result for left composition can be found in \cite{HJR}.
\begin{Prop}\label{Composition}
Let $p\geq 1$ and $f$ be a function in $\mathcal{C}^\infty(\mathbb{R})$ such that $f(0)=0$. let $(s_1,s_2)\in]0,\infty[^2$ and $(r_1,r_2)\in[1,\infty]^2$. We assume that $s_1<{d}/{p}$ or that $s_1={d}/{p}$ and $r_1=1$.

Then, for every real-valued function $u$ in $\dot{\mathbb{B}}^{s_1}_{p,r_1}\cap\dot{\mathbb{B}}^{s_2}_{p,r_2}\cap L^\infty$, the function $f\circ u$ belongs to $\dot{\mathbb{B}}^{s_1}_{p,r_1}\cap\dot{\mathbb{B}}^{s_2}_{p,r_2}\cap L^\infty$ and we have
$$\norme{f\circ u}_{\dot{\mathbb{B}}^{s_k}_{p,r_k}}\leq C\left(f',\norme{u}_{L^\infty}\right)\norme{u}_{\dot{\mathbb{B}}^{s_k}_{p,r_k}}\quad\hbox{for}\  k\in\{1,2\}.$$
\end{Prop}
As a consequence (see \cite[Cor. 2.66]{HJR}), if $g$ is a $\cC^\infty(\R)$ function such that $g'(0)=0$, then,  
for all $u,v$ in $\dot\B^s_{p,1}\cap L^\infty$ with $s>0,$ we have
\begin{equation}\label{eq:compo}
\|g(v)-g(u)\|_{\dot\B^s_{p,1}} \leq C\Bigl(\|v-u\|_{L^\infty}\|(u,v)\|_{\dot\B^s_{p,1}} + 
\|v-u\|_{\dot\B^s_{p,1}} \|(u,v)\|_{L^\infty}\Bigr)\cdotp
\end{equation}

We also need the following more involved product law to handle 
the high frequencies of some non-linear terms.
 \begin{Prop} \label{LPP} 
Let  $2\leq p\leq 4$ and   $p^*\triangleq 2p/(p-2).$ 
For all $\sigma \geq s>0$, we have
\begin{equation}\label{eq:prod4}
\|ab\|^h_{\dot\B^{s}_{2,1}}\lesssim  \|a\|_{\dot\B^{\frac dp}_{p,1}}\|b\|^h_{\dot\B^{s}_{2,1}}
+\|b\|_{\dot\B^{\frac dp}_{p,1}}\|a\|^h_{\dot\B^{s}_{2,1}}+ \|a\|^\ell_{\dot\B^{\frac dp-\frac{d}{p*}}_{p,1}}\|b\|^\ell_{\dot\B^{\sigma}_{p,1}}
+ \|b\|^\ell_{\dot\B^{\frac dp-\frac{d}{p*}}_{p,1}}\|a\|^\ell_{\dot\B^{\sigma}_{p,1}}.
\end{equation}
\end{Prop}
\begin{proof} 
Recall the  following so-called Bony decomposition 
(first introduced by J.-M. Bony in \cite{Bony})
for the product 
of two tempered distributions $f$ and $g$:
$$fg=T_fg+T'_gf\with T_fg\triangleq\sum_{j\in\Z}\dot S_{j-1}f\,\ddj g\andf
T'_gf\triangleq \sum_{j\in\Z}\dot S_{j+2}g\,\ddj f.$$
Using this decomposition 
 and further splitting  $a$ and $b$ into low and high frequencies,
 we get
 $$ ab=T_{a^\ell}b^\ell+T'_{b^\ell}a^\ell+T'_{b}a^h+T_{a}b^h+T'_{b^h}a^\ell+T_{a^h}b^\ell.$$
 All the terms in the right-hand side, except for the last two ones, may be bounded
 by means of standard results of continuity for operators $T$ and $T'$
   (see again \cite[Chap. 2]{HJR}). Provided $\sigma\geq s>0,$ we get,
 $$\begin{aligned}
 \|T'_{b^\ell}a^\ell\|^h_{\dot\B^{s}_{2,1}}& \lesssim \|T'_{b^\ell}a^\ell\|^h_{\dot\B^{\sigma}_{2,1}}\lesssim \|b^\ell\|_{L^{p^*}}\|a^\ell\|_{\dot\B^{\sigma}_{p,1}},\\
 \|T'_{b}a^h\|_{\dot\B^{s}_{2,1}}&\lesssim\|b\|_{L^\infty}\|a^h\|_{\dot\B^{s}_{2,1}}.
\end{aligned}$$ 
Let $J_1$ be the integer corresponding to the threshold 
between low and high frequencies. Since $a^\ell=\dot S_{J_1+1}a$ and $b^h=({\rm Id}-\dot S_{J_1+1})b,$ we see that 
$$
T'_{b^h} a^\ell=  \dot S_{J_1+2}b^h\,\dot\Delta_{J_1+1} a^\ell.$$
Consequently, as 
$\dot S_{J_1+2}b^h=(\dot\Delta_{J_1-1}+\dot\Delta_{J_1}+\dot\Delta_{J_1+1})b^h,$
$$\|T_{b^h} a^\ell\|_{\dot\B^{s}_{2,1}}\lesssim
\|\dot\Delta_{J_1+1}a^\ell\|_{L^\infty} \|\dot S_{J_1+2}b^h\|_{L^2}\lesssim \|a\|_{L^\infty} 
 \|b\|^h_{\dot\B^{s}_{2,1}}.$$
Adding up this latter inequality to the previous one and to the symmetric ones (with just operator $T$ instead of $T'$),  and the embeddings 
   $\dot\B^{\frac dp}_{p,1}\hookrightarrow L^\infty$ and, as $p\leq p*$,
   $\dot\B^{\frac dp-\frac d{p*}}_{p,1}\hookrightarrow L^{p^*}$ completes the proof of \eqref{eq:prod4}. 
\end{proof}

 To handle commutators  in high frequencies, we need the following lemma.
  \begin{Lemme}
\label{CP}Let $p\in[2,4]$ and $s>0$. Define $p^*\triangleq 2p/(p-2).$  For $j\in\Z,$ 
denote $\mathfrak{R}_j\triangleq \dot S_{j-1}w\,\ddj z-\ddj(wz)$.

There exists  a constant $C$ depending only on  the threshold 
number $J_1$ between low and high frequencies and on $s,$ $p,$ $d,$ such that

$$\displaylines{
\sum_{j\geq J_1}\left(2^{js}\norme{\mathfrak{R}_j}_{L^2}\right)\leq C\Bigl(\norme{\nabla w}_{\dot{\mathbb{B}}^{\frac dp}_{p,1}}\norme{z}^h_{\dot{\mathbb{B}}^{s-1}_{2,1}}
+ \norme{z}^\ell_{\dot{\mathbb{B}}^{\frac{d}{p}-\frac{d}{p*}}_{p,1}}\norme{w}^\ell_{\dot{\mathbb{B}}^{\sigma_1}_{p,1}}\hfill\cr\hfill+\norme{z}_{\dot{\mathbb{B}}^{\frac dp-k}_{p,1}}\norme{w}^h_{\dot{\mathbb{B}}^{s+k}_{2,1}}
+ \norme{z}^\ell_{\dot{\mathbb{B}}^{\sigma_2}_{p,1}}\norme{ \nabla  w}^\ell_{\dot{\mathbb{B}}^{\frac{d}{p}-\frac{d}{p*}}_{p,1}}\Bigr),}$$
for any $k\geq0$, $\sigma_1 \geq s$ and $\sigma_2\in\R.$


\end{Lemme}
\begin{proof}
From Bony's decomposition recalled above
and the fact that $\dot\Delta_j\dot\Delta_{j'}=0$ for $|j-j'|\geq2,$
we deduce that
$$\begin{aligned}\mathfrak{R}_j&=-\ddj(T'_{z}w)-\sum_{|j'-j|\leq4} [\ddj,\dot S_{j'-1}w]\dot\Delta_{j'}z
- \sum_{|j'-j|\leq 1}\left(\dot{S}_{j'-1}w-\dot{S}_{j-1}w\right)\dot{\Delta}_j\dot{\Delta}_{j'} z\\
&\triangleq\mathcal{R}^1_j + \mathcal{R}^2_j +  \mathcal{R}^3_j. 
\end{aligned}$$
To estimate  $\mathcal{R}^1_j$, we  use the decomposition
$$
T'_zw=T'_{z^\ell}w^\ell +T'_{z^h}w^\ell+T'_zw^h$$
and proceed as in the proof of Proposition \ref{LPP}. In the end, we get
$$\|T'_{z}w\|^h_{\dot\B^{s}_{2,1}}\lesssim\|z\|_{\dot\B^{-k}_{\infty,1}}\|w\|^h_{\dot\B^{s+k}_{2,1}}+ \|z\|^\ell_{\dot\B^{\frac dp-\frac{d}{p*}}_{p,1}}\|w\|^\ell_{\dot\B^{\sigma_1}_{p,1}}.$$
Therefore, since $\dot\B^{\frac dp-k}_{p,1}\hookrightarrow\dot\B^{-k}_{\infty,1},$
\begin{equation}\label{eq:Rj1}
\sum_{j\in\Z}\left(2^{js}\norme{\cR_j^1}_{L^2}\right)\lesssim
\|z\|_{\dot\B^{\frac{d}{p}-k}_{p,1}}\|w\|^h_{\dot\B^{s+k}_{2,1}}+ \|z\|^\ell_{\dot\B^{\frac dp-\frac{d}{p*}}_{p,1}}\|w\|^\ell_{\dot\B^{\sigma_1}_{p,1}}.
\end{equation}
Next, taking advantage of Lemma \ref{C1}, we see that
if $j'\geq J_1$ and $|j-j'|\leq 4,$ then we have 
$$
2^{js} \|[\ddj,\dot S_{j'-1}w] \dot\Delta_{j'}z\|_{L^2}
\lesssim \|\nabla \dot S_{j'-1}w\|_{L^{\infty}} \, 2^{j'(s-1)}\| \dot\Delta_{j'}z\|_{L^2}
$$
while, if $j'<J_1,$ $j\geq J_1$ and $|j-j'|\leq 4,$ 
\begin{eqnarray*}
2^{js} \|[\ddj,\dot S_{j'-1}w]\dot\Delta_{j'}z\|_{L^2}
&\lesssim& 2^{J_1(s-\sigma_2-1)}2^{j(\sigma_2+1)} \|[\ddj,\dot S_{j'-1}w]\dot\Delta_{j'}z\|_{L^2}\\&\lesssim& 2^{J_1(s-\sigma_2-1)}\|\nabla\dot S_{j'-1}w\|_{L^{p^*}} \, 2^{j'\sigma_2}\| \dot\Delta_{j'}z\|_{L^p}.
\end{eqnarray*}
Therefore, 
\begin{eqnarray}\label{eq:Rj2}
\sum_{j\geq J_1}\left(2^{js}\norme{\cR_j^2}_{L^2}\right)&\lesssim &
\norme{z}^h_{\dot{\mathbb{B}}^{s-1}_{2,1}}\norme{\nabla w}_{L^\infty}+\norme{z}^\ell_{\dot{\mathbb{B}}^{\sigma_2}_{p,1}}\norme{\nabla w}_{L^{p^*}}^\ell\cdotp
\end{eqnarray}
Then  with suitable embeddings, one gets
\begin{eqnarray}
\label{eq:Rj22}
\sum_{j\geq J_1}\left(2^{js}\norme{\cR_j^2}_{L^2}\right)&\lesssim &
\norme{z}^h_{\dot{\mathbb{B}}^{s-1}_{2,1}}\norme{\nabla w}_{\dot{\mathbb{B}}^{\frac{d}{p}}_{p,1}}+\norme{z}^\ell_{\dot{\mathbb{B}}^{\sigma_2}_{p,1}}\norme{\nabla w}_{\dot{\mathbb{B}}^{\frac{d}{p}-\frac{d}{p*}}_{p,1}}^\ell\cdotp
\end{eqnarray}
Finally, for all $j\geq J_1$ and $|j'-j|\leq 1,$ we have
$$\begin{aligned}
2^{js}\|(\dot{S}_{j'-1}w-\dot{S}_{j-1}w)\dot{\Delta}_j\dot{\Delta}_{j'}z\|_{L^2}
&\leq 2^j\|\dot\Delta_{j\pm1} w\|_{L^\infty}\, 2^{j(s-1)}\| \dot\Delta_{j'}\ddj z\|_{L^2}\\
&\leq C\|\dot\Delta_{j\pm1} \nabla w\|_{L^\infty}\, 2^{j(s-1)}\| \ddj z\|_{L^2}.\end{aligned}
$$
Hence
\begin{equation}\label{eq:Rj3}
\sum_{j\geq J_1}\left(2^{js}\norme{\cR_j^3}_{L^2}\right)\leq 
C\|\nabla w\|_{L^\infty}\|z\|^h_{\dot\B^{s-1}_{2,1}}.\end{equation}
Putting \eqref{eq:Rj1}, \eqref{eq:Rj22} and \eqref{eq:Rj3} together  yields the desired estimate.

\end{proof}

\bibliographystyle{plain}


\bibliography{Biblio}

\begin{thebibliography}{10}

\bibitem{HJR}
H.~Bahouri, J.-Y. Chemin, and R.~Danchin.
\newblock {\em Fourier Analysis and Nonlinear Partial Differential Equations},
  volume 343 of {\em Grundlehren der Mathematischen Wissenschaften}.
\newblock Springer, Heidelberg, 2011.

\bibitem{BZ}
K.~Beauchard and E.~Zuazua.
\newblock Large time asymptotics for partially dissipative hyperbolic systems.
\newblock {\em Arch. Rational Mech. Anal}, 199, 177–227, 2011.

\bibitem{BHN}
S.~Bianchini, B.~Hanouzet, and R.~Natalini.
\newblock Asymptotic behavior of smooth solutions for partially dissipative
  hyperbolic systems with a convex entropy.
\newblock {\em Comm. Pure and Appl. Math.}, 60, 1559-1622, 2007.

\bibitem{Bony}
J.-M. Bony.
\newblock Calcul symbolique et propagation des singularités pour les
  équations aux dérivées partielles non linéaires.
\newblock {\em Ann. Sci. École Norm. Sup.}, (4) 14:209--246, 1981.

\bibitem{NSCLP}
F.~Charve and R.~Danchin.
\newblock A global existence result for the compressible {N}avier-{S}tokes
  equations in the critical $\textsc{L}^p$ framework.
\newblock {\em Arch. Rational Mech. Anal}, 198, 233-271, 2010.

\bibitem{chen1994}
G-Q. Chen, C.D. Levermore, and T-P. Liu.
\newblock Hyperbolic conservation laws with stiff relaxation terms and entropy.
\newblock {\em Communications on Pure and Applied Mathematics}, 47(6):787--830,
  1994.

\bibitem{CMZ}
Q.~Chen, C.~Miao, and Z.~Zhang.
\newblock Global well-posedness for compressible {N}avier-{S}tokes equations
  with highly oscillating initial velocity.
\newblock {\em Comm. Pure Appl. Math.}, 63(9):1173--1224, 2010.

\bibitem{Coron}
J.-M. Coron.
\newblock {\em Control and nonlinearity}, volume 136.
\newblock Mathematical Surveys and Monographs. American Mathematical Society,
  2007.

\bibitem{CoulombelGoudon}
J.-F. Coulombel and T.~Goudon.
\newblock The strong relaxation limit of the multidimensional isothermal
  {E}uler equations.
\newblock {\em Trans. Amer. Math. Soc.}, 359(2):637–648, 2007.

\bibitem{CBD2}
T.~Crin-Barat and R.~Danchin.
\newblock Partially dissipative hyperbolic systems in the critical regularity
  setting : the multi-dimensional case.
\newblock {\em arXiv:2105.08333}, 2021.

\bibitem{CBD1}
T.~Crin-Barat and R.~Danchin.
\newblock Partially dissipative one-dimensional hyperbolic systems in the
  critical regularity setting, and applications.
\newblock {\em Pure and Applied Analysis}, to appear.

\bibitem{Da01}
R.~Danchin.
\newblock Local theory in critical spaces for compressible viscous and
  heat-conductive gases.
\newblock {\em Comm. Partial Differential Equations}, 26(7-8):1183--1233, 2001.

\bibitem{hsiao1}
L.~Hsiao.
\newblock {\em Quasilinear Hyperbolic Systems and Dissipative Mechanisms},
  volume Singapore.
\newblock World Scientific Publishing, 1997.

\bibitem{Junca}
S.~Junca and M.~Rascle.
\newblock Strong relaxation of the isothermal {E}uler system to the heat
  equation.
\newblock {\em Z. angew. Math. Phys.}, 53, 239–264, 2002.

\bibitem{Kawa1}
S.~Kawashima.
\newblock Systems of a hyperbolic-parabolic composite type, with applications
  to the equations of magnetohydrodynamics.
\newblock {\em Doctoral Thesis}, 1983.

\bibitem{KYDecay}
S.~Kawashima and W.-A. Yong.
\newblock Decay estimates for hyperbolic balance laws.
\newblock {\em Journal for Analysis and its Applications}, 28, 1–33, 2009.

\bibitem{LiTT}
T-T. Li.
\newblock {\em Global classical solutions for quasilinear hyperbolic systems}.
\newblock Masson, Paris; John Wiley \& Sons, Ltd., Chichester, 1994.

\bibitem{Peng2019}
Y.~Li, Y-J. Peng, and L.~Zhao.
\newblock Convergence rate from hyperbolic systems of balance laws to parabolic
  systems.
\newblock {\em Applicable Analysis}, pages 1079--1095, 2021.

\bibitem{LiangShuai}
Z.~Liang and Z.~Shuai.
\newblock Convergence rate from hyperbolic systems of balance laws to parabolic
  systems.
\newblock {\em Asymptotic Analysis}, 1224:163--198, 2021.

\bibitem{CoulombelLin}
C.~Lin and J.-F. Coulombel.
\newblock The strong relaxation limit of the multidimensional {E}uler
  equations.
\newblock {\em Nonlinear Differential Equations and Applications NoDEA},
  20:447--461, 2013.

\bibitem{Majda}
A.~Majda.
\newblock {\em Compressible Fluid Flow and Systems of Conservation Laws in
  Several Space Variable}.
\newblock Springer, New-York, 1984.

\bibitem{mar0}
P.~Marcati and A.~Milani.
\newblock The one-dimensional {D}arcy’s law as the limit of a compressible
  {E}uler flow.
\newblock {\em J. Differential Equations}, 84, 129-147, 1990.

\bibitem{mar1}
P.~Marcati and B.~Rubino.
\newblock Hyperbolic to parabolic relaxation theory for quasilinear first order
  systems.
\newblock {\em J. Differential Equations}, 162, 359-399, 2000.

\bibitem{WasiolekPeng}
Y.-J. Peng and V.~Wasiolek.
\newblock Parabolic limit with differential constraints of first-order
  quasilinear hyperbolic systems.
\newblock {\em Ann. I. H. Poincaré}, 33(4):1103--1130, 2016.

\bibitem{Serre}
D.~Serre.
\newblock {\em Systèmes de lois de conservation, tome 1}.
\newblock Diderot editeur, Arts et Sciences, Paris, New-York, Amsterdam, 1996.

\bibitem{SK}
S.~Shizuta and S.~Kawashima.
\newblock Systems of equations of hyperbolic-parabolic type with applications
  to the discrete {B}oltzmann equation.
\newblock {\em Hokkaido Math. J.}, 14, 249-275, 1985.

\bibitem{Villani}
C.~Villani.
\newblock Hypocoercivity.
\newblock {\em Mem. Am. Math. Soc.}, 2010.

\bibitem{ThesisWasiolek}
V.~Wasiolek.
\newblock Analyse asymptotique de systèmes hyperboliques quasi-linéaires du
  premier ordre.
\newblock {\em Thesis dissertation, Université Blaise Pascal -
  Clermont-Ferrand II}, 2015.

\bibitem{XK1E}
J.~Xu and S.~Kawashima.
\newblock Diffusive relaxation limit of classical solutions to the damped
  compressible {E}uler equations.
\newblock {\em Journal of Differential Equations}, 256, 771-796, 2014.

\bibitem{XK1}
J.~Xu and S.~Kawashima.
\newblock Global classical solutions for partially dissipative hyperbolic
  system of balance laws.
\newblock {\em Arch. Rational Mech.Anal}, 211, 513–553, 2014.

\bibitem{XuWang}
J.~Xu and Z.~Wang.
\newblock Relaxation limit in besov spaces for compressible {E}uler equations.
\newblock {\em Journal de Mathématiques Pures et Appliquées}, 99:43--61,
  2013.

\bibitem{Yong}
W.-A Yong.
\newblock Entropy and global existence for hyperbolic balance laws.
\newblock {\em Arch. Rational Mech. Anal}, 172, 47–266, 2004.

\bibitem{SlideZuazua}
E.~Zuazua.
\newblock Decay of partially dissipative hyperbolic systems.
\newblock {\em https://caa-avh.nat.fau.eu/enrique-zuazua-presentations/}, 2020.

\end{thebibliography}
\end{document}